\documentclass[hidelinks,onefignum,onetabnum]{siamonline220329}

\usepackage{lipsum}
\usepackage{epstopdf}
\usepackage{booktabs}
\usepackage{amsmath,amsfonts}
\usepackage{algorithm}
\usepackage{algpseudocode}
\usepackage{array}
\usepackage[caption=false,font=normalsize,labelfont=sf,textfont=sf]{subfig}
\usepackage{textcomp}
\usepackage{stfloats}
\usepackage{url}
\usepackage{verbatim}
\usepackage{graphicx}
\usepackage{cite}
\usepackage{acro}
\usepackage{xcolor}
\usepackage{soul}
\usepackage{amssymb}
\usepackage{mathtools}
\usepackage{mathrsfs}
\usepackage{dsfont}
\usepackage{enumerate}
\usepackage{tikz}

\newsiamremark{assumption}{Assumption}

\DeclareAcronym{NF}{
  short = NF,
  long = normalizing flow
  }

\ifpdf
  \DeclareGraphicsExtensions{.eps,.pdf,.png,.jpg}
\else
  \DeclareGraphicsExtensions{.eps}
\fi
\usepackage[shortlabels]{enumitem}
\setlist[enumerate]{leftmargin=.5in}
\setlist[itemize]{leftmargin=.5in}

\headers{ULA using normalizing flow prior}{Z. Cai, J. Tang, S. Mukherjee, J. Li, C.-B. Sch\"onlieb, and X. Zhang}

\title{NF-ULA: Normalizing flow-based unadjusted Langevin algorithm for imaging inverse problems}

\author{
Ziruo Cai\thanks{School of Mathematical Sciences, Shanghai Jiao Tong University, China (\href{mailto:sjtu_caiziruo@sjtu.edu.cn}{sjtu\_caiziruo@sjtu.edu.cn}).}
\and
Junqi Tang\thanks{School of Mathematics, University of Birmingham, UK (\email{j.tang.2@bham.ac.uk}).}
\and
Subhadip Mukherjee\thanks{Department of Electronics and Electrical Communication Engineering, Indian Institute of Technology (IIT) Kharagpur, India. (\email{smukherjee@ece.iitkgp.ac.in}).}
\and
Jinglai Li\thanks{School of Mathematics, University of Birmingham, UK (\email{j.li.10@bham.ac.uk}).}
\and
Carola-Bibiane Sch\"onlieb\thanks{Department of Applied Mathematics and Theoretical Physics, University of Cambridge, UK (\email{cbs31@cam.ac.uk}).}
\and
Xiaoqun Zhang\thanks{School of Mathematical Sciences, MOELSC and Institute of Natural Sciences, Shanghai Jiao Tong University,  China(\email{xqzhang@sjtu.edu.cn}).}
}

\usepackage{amsopn}

\usepackage{acro}
\usepackage{enumitem}

\ifpdf
\hypersetup{
  pdftitle={NF-ULA: Normalizing flow-based unadjusted Langevin algorithm for imaging inverse problems},
  pdfauthor={Ziruo Cai, Junqi Tang, Subhadip Mukherjee, Jinglai Li, Carola-Bibiane Sch\"onlieb, and Xiaoqun Zhang}
}
\fi

\begin{document}

\maketitle

\begin{abstract}
Bayesian methods for solving inverse problems are a powerful alternative to classical methods since the Bayesian approach offers the ability to quantify the uncertainty in the solution. In recent years, data-driven techniques for solving inverse problems have also been remarkably successful, due to their superior representation ability. In this work, we incorporate data-based models into a class of Langevin-based sampling algorithms for Bayesian inference in imaging inverse problems. In particular, we introduce NF-ULA (Normalizing Flow-based Unadjusted Langevin algorithm), which involves learning a \textit{normalizing flow} (NF) as the image prior. We use NF to learn the prior because a tractable closed-form expression for the log prior enables the differentiation of it using \textit{autograd} libraries. Our algorithm only requires a normalizing flow-based generative network, which can be pre-trained independently of the considered inverse problem and the forward operator. We perform theoretical analysis by investigating the well-posedness and non-asymptotic convergence of the resulting NF-ULA algorithm. The efficacy of the proposed NF-ULA algorithm is demonstrated in various image restoration problems such as image deblurring, image inpainting, and limited-angle X-ray computed tomography (CT) reconstruction. NF-ULA is found to perform better than competing methods for severely ill-posed inverse problems.

\end{abstract}



\section{Introduction}

Imaging inverse problems can be formulated as
$y = Ax  + n  , $
where $y\in \mathbb{R}^m$ is the indirect noisy observation, $A: \mathbb{R}^d\rightarrow\mathbb{R}^m$ is the observation operator, $n$ is the measurement noise, and $x\in\mathbb{R}^d$ represents the unknown image that one aims to recover. In the classical variational framework, the reconstruction problem is formulated as the minimization of an energy functional $J(x)=L(y,Ax)+\alpha\,g(x)$, where $L$ measures data-consistency and $g$ is a regularizer that penalizes undesirable images. Following the surge of deep learning, data-driven regularization methods have become ubiquitous in imaging inverse problems \cite{benning2018modern, arridge2019solving, ongie2020deep}, leading to state-of-the-art results which significantly outperform classical hand-crafted regularization schemes such as the total-variation \cite{chambolle2010introduction} or sparsity-based regularizers (see \cite{benning2018modern} and references therein). Starting from the plug-and-play methods \cite{venkatakrishnan2013plug} which combine proximal-splitting optimization algorithms \cite{combettes2011proximal} with learned denoisers \cite{ronneberger2015u, jin2017deep, zhang2021plug}, researchers have made considerable progress in this direction. Current popular trends in this line of research include the studies in improving practical performances and theoretical guarantees \cite{ romano2017little,hertrich2021convolutional, gilton2019learned,sreehari2016plug,kamilov2023plug, tan2023provably}, the development of deep unrolling networks \cite{monga2021algorithm, adler2018learned}, deep equilibrium models \cite{gilton2021deep}, the studies on the image prior by specific  networks structures \cite{lempitsky2018deep}, the extension of generative models in imaging applications \cite{song2021solving, pan2021exploiting, asim2020invertible, whang2021solving}, operator regularization methods \cite{pesquet2021learning}, learning explicitly the regularization functional such as a gradient-step denoiser \cite{hurault2021gradient}, total deep variation \cite{kobler2021total}, adversarial regularizers \cite{lunz2018adversarial, mukherjee2021end, prost2021learning} and the learned convex regularizer \cite{mukherjee2020learned} with input-convex neural networks \cite{amos2017input}.

While the previously mentioned approaches treat $x$ deterministically, another alternative framework for solving inverse problems is to do it within a Bayesian setting \cite{stuart2010inverse, kaipio2006statistical, tarantola2005inverse}. Different from the functional-analytic methods, Bayesian methods model the image $x$ as a random variable and usually seek to approximate the posterior distribution $p(x | y)$ based on Bayes' formula. The methods based on Bayesian inference can not only give a point estimator (e.g., the maximum a posteriori probability (MAP) estimator) but also describe the uncertainty in the solution in a probabilistic way in terms of variance and credible intervals. The capability of uncertainty quantification is particularly helpful for decision-making and reliability assessment. Typical examples of Bayesian imaging schemes include the classical approach using the total variation prior \cite{louchet2013posterior, pereyra2016proximal}, the works on Markov random fields \cite{blake2011markov}, and more recently the patch-based models \cite{zoran2011learning, yu2011solving, aguerrebere2017bayesian, houdard2018high}.

In Bayesian inference, one explores the posterior distribution to generate samples from them, typically using the Markov Chain Monte Carlo (MCMC) methods~\cite{gilks1995markov}. Among these sampling algorithms, the Langevin Monte Carlo (LMC) algorithms \cite{neal1992bayesian, roberts1996exponential}, also referred to as the \textit{Unadjusted Langevin Algorithms} (ULA), stand out as an increasingly popular tool, since they bridge the gap between theoretical guarantees of nonasymptotic convergence analysis\cite{durmus2017nonasymptotic, dalalyan2017theoretical, de2019convergence} and practical performance\cite{durmus2018efficient, laumont2022bayesian}.  ULA can also be modified into Metropolis-adjusted Langevin algorithm (MALA) \cite{roberts1996exponential}, a non-biased version, by adding a Metropolis-Hastings (MH) accept-reject step. Apart from the MCMC-based methods, there are also other kinds of sampling methods worth mentioning: methods based on variational inference \cite{hoffman2013stochastic, blei2017variational, liu2016stein} posit a family of densities and then attempt to find a member of that family which is close to the target density. Variational auto-encoders (VAEs) \cite{kingma2013auto} approximate the posterior by learning deep encoders and decoders. Generative adversarial networks (GAN) \cite{goodfellow2014generative, creswell2018generative} learn the generator to sample from the training distribution through adversarial learning. More recently, diffusion models \cite{song2020score, ho2020denoising, yang2022diffusion} have been shown to be a powerful tool for image generation. They learn the target distribution by transforming an image into a Gaussian noise and then reversing the noising process.

In recent years, the theoretical analysis and nonasymptotic convergence of ULA \cite{dalalyan2017theoretical, durmus2017nonasymptotic} have opened a new direction of research. Except for convex and smooth potentials \cite{dalalyan2017theoretical, durmus2017nonasymptotic, dalalyan2019user, durmus2019analysis}, ULA for non-convex or non-smooth potentials has also seen great progress. While ULA requires evaluating the gradient of potentials, ULA for non-smooth distributions \cite{pereyra2016proximal, durmus2018efficient,  salim2019stochastic, lehec2021langevin, luu2021sampling, mou2022efficient} draw samples from a smoothed proxy, borrowing the tools such as proximity operators from non-smooth optimization literature. For non-convex potentials, ULA also has convergence guarantees \cite{cheng2018sharp, de2019convergence, majka2020nonasymptotic, erdogdu2021convergence} if some conditions, (e.g., contractivity condition on the drift) are satisfied.

Incorporating data-based approaches into classical algorithms is a trending topic in ULA and Bayesian methods for solving inverse problems. More specifically, one aims to utilize an over-parameterized model learned on given data, such as a neural network, instead of handcrafting the prior. Recently, Langevin Monte Carlo using Plug and Play Prior (PnP-ULA) \cite{laumont2022bayesian} was shown to yield promising results for Bayesian imaging problems. PnP-ULA leverages an implicit image prior learned via a plug-and-play (PnP) \cite{venkatakrishnan2013plug} Lipschitz-continuous image denoiser \cite{ryu2019plug}. Since the true image prior is not assumed to be convex or smooth, PnP-ULA convergence was established for non-convex potentials.

Besides PnP priors \cite{venkatakrishnan2013plug}, normalizing flow (NF)-based approaches \cite{rezende2015variational, dinh2016density, papamakarios2021normalizing} also lead to impressive performance on imaging problems \cite{kingma2018glow, dinh2016density} and have the potential of learning the prior in the Bayesian imaging framework. In this work, we attempt to integrate an image prior that is learned by NF into the Langevin algorithms. Notably, the resulting negative log posterior in our case is non-convex. To make the model well-defined in the Bayesian setting and to ensure that the algorithm is numerically stable, we make minor changes to the standard ULA to add a projection-minus-identity operator on a compact set, akin to PnP-ULA \cite{laumont2022bayesian}. 
As some studies of normalizing flows have shown\cite{rezende2015variational, dinh2016density, papamakarios2021normalizing, kingma2018glow},  training a normalizing flow prior for natural images generally requires utilizing larger networks,  larger training dataset, more computational resources and more time than training a PnP denoiser, our proposed method is more efficient if the normalizing flow prior is pre-trained and available.

The idea of interlacing NF with MCMC algorithms has been considered previously in the literature, but these methods had significant conceptual differences from our approach. For instance, \cite{wu2020stochastic} proposed stochastic NF, an arbitrary sequence of deterministic invertible functions and stochastic sampling blocks, to sample from target density. The authors of \cite{steidl2022generalized, hagemann2022stochastic} considered stochastic NF from a Markov chain point of view and replaced the transition densities with general Markov kernels. \cite{coeurdoux2023normalizing} utilized NF to sample from the target distribution in the latent domain before transporting it back to the target domain relying on MALA. There are some studies combining other  generative models with non-Langevin Monte Carlo algorithms, e.g., \cite{coeurdoux2023plug} introduced a stochastic PnP sampling algorithm leveraging variable splitting to efficiently sample from a posterior distribution using  diffusion-based
generative models\cite{dhariwal2021diffusion}. To summarize, all the above mentioned approaches are different from ours, mainly because they do not directly utilize the log gradient density of NF in Langevin algorithms.

\noindent\subsection{Our contributions} The main contributions of this work are:
\begin{enumerate}
\item We propose NF-ULA, a novel framework of sampling by Langevin Monte Carlo-based algorithms while leveraging a pre-trained normalizing flow induced prior. Since both the density and the log gradient of the density of normalizing flows can be evaluated, NF-ULA can potentially be extended to a Metropolis-adjusted version. 
\item We give a sufficient condition to ensure the Lipschitz gradient of the log density of the normalizing flows since the Lipschitz gradient is one of the most essential conditions to guarantee the convergence of ULA. This might also be useful in the future when an NF-based prior is used in methods other than Langevin algorithms. 
\item We show that the Bayesian solution of NF-ULA is well-defined and well-posed and establish that NF-ULA admits an invariant distribution. We also give a non-asymptotic bound on the bias. 
\item We demonstrate that NF-ULA yields high-quality results in applications such as image deblurring, image inpainting, and limited-angle X-ray computed tomography (CT) reconstruction. For more ill-posed problems, NF-ULA demonstrates stronger regularization than competing methods. We also provide experimental evidence that enhanced training of the NF prior results in improved sampling and reconstruction, especially for severely ill-posed problems (such as limited-angle CT).
\end{enumerate}

The rest of the paper is organized as follows: Sec. \ref{sec:problem_setup} gives a brief review of both Langevin Monte Carlo and normalizing flow, leading to the proposed NF-ULA method. Sec. \ref{sec:theoretical_analysis} presents a theoretical analysis of the Bayesian solution obtained using NF-ULA. In Sec. \ref{sec:Experiments}, we evaluate NF-ULA on image deblurring, image inpainting, and limited-angle CT reconstruction. Final conclusions are summarized in Sec. \ref{sec:conclusions}. The proofs and extra experiments are in the Appendix.

\section{Mathematical background and the proposed method}\label{sec:problem_setup}
We begin by giving some background on Langevin Monte Carlo (LMC) algorithms and normalizing flow. Subsequently, we propose NF-ULA, an LMC algorithm that utilizes a pre-trained normalizing flow network.
\subsection{LMC for Non-smooth Potentials}\label{sec:LMC}
In Bayesian inference, there is a broad class of problems where we seek to draw samples $ \{X_k\}_{k = 1}^{K}, X_k\in \mathbb{R}^d$, from a target posterior distribution $p(x | y)$, given the observation $y \in \mathbb{R}^m$. Using Bayes' formula, we have that
\begin{equation}
\label{Bayes formula}
p (x | y)= \dfrac{p(y | x) p(x)}{\int p(y | \tilde{x}) p(\tilde{x}) \mathrm{d} \tilde{x}}.
\end{equation}
Under some assumptions on the likelihood $p(y|x)$ and the prior $p(x)$, the posterior distribution $p(x|y)$ is well-posed; meaning that it is well-defined, unique, and varies continuously in $y$ with respect to appropriate distance metrics for probability distributions\cite{sprungk2020local, latz2020well}. The well-known LMC approach \cite{neal1992bayesian, roberts1996exponential}, also referred to as the \textit{unadjusted Langevin algorithm} (ULA), can efficiently sample from $p(x | y)$ using the following Markov chain:
\begin{equation}
\label{ULA}
\begin{aligned}
    X_{k+1} &= X_k+\delta \nabla \log p \left( X_k | y \right) + \sqrt{2 \delta} Z_{k+1} \\ 
    & = X_k+\delta \nabla \log p\left(y | X_k\right)+\delta \nabla \log p\left(X_k\right)+\sqrt{2 \delta} Z_{k+1},    
\end{aligned}    
\end{equation}
where $\{ Z_k \}_k \sim \mathcal{N}(0, I^d)$ is a family of i.i.d. standard Gaussian random variables. The ULA approach in \eqref{ULA} is based on the Euler-Maruyama (EM) discretization with step-size $\delta$ of the over-damped Langevin stochastic differential equation (SDE) given by
\begin{equation}
    \mathrm{d} X_t=\nabla \log p\left(X_t | y\right)\mathrm{d} t + \sqrt{2} \mathrm{~d} B_t,
    \label{eq:LSDE_posterior}
\end{equation}
where $B_t$ is a Brownian motion. It has been shown in \cite{dalalyan2017theoretical, durmus2017nonasymptotic} that when $-\log p(x| y)$ is continuously differentiable and has Lipschitz gradient,    the convergence of ULA can be guaranteed if convexity of $-\log p(x| y)$ \cite{dalalyan2017theoretical} or contractivity in the tails \cite{durmus2017nonasymptotic} is satisfied. The convergence is subject to a bias related to the step-size $\delta$. In general, smaller $\delta$ leads to a smaller bias and larger $\delta$ leads to faster convergence of the Markov Chain. The non-asymptotic bias and convergence analysis of ULA have remained relatively under-explored until the last few years \cite{dalalyan2017theoretical, durmus2017nonasymptotic, dalalyan2019user, durmus2019analysis}. Notably, the bias of ULA in \eqref{ULA} can be removed by adding a Metropolis-Hastings (MH) accept-reject step, leading to the so-called Metropolis-adjusted Langevin algorithm (MALA) \cite{roberts1996exponential}. In this paper, we will focus on ULA without any MH adjustments.

When the potential $-\log p(x)$ is convex but non-smooth, \cite{durmus2018efficient}  uses a smooth proxy utilizing the Moreau envelope $U^{(\lambda)}(x)$ of $U(x) = -\log p(x)$ in \eqref{ULA}. The Moreau envelope $U^{(\lambda)}(x)$ and the proximity operator $\operatorname{prox}_{\lambda, U}$ of $U(x)$ are defined as
\begin{equation*}
\begin{aligned}
& {U}^{(\lambda)}(x) :=  \underset{z\in \mathbb{R}^d}{\inf}\, \left( U(z)+\frac{1}{2\lambda}\|x-z\|_2^2\right),\text{\,\,and}
&\operatorname{prox}_{\lambda,U} (x) := \underset{z\in \mathbb{R}^d}{\arg\min}\, \left( U(z)+\frac{1}{2\lambda}\|x-z\|_2^2\right) .
\end{aligned}
\end{equation*}
For a convex function $U$, $\operatorname{prox}_{\lambda,U} (x)$ is unique and well-defined.

Since the Moreau envelope $U^{(\lambda)}(x)$ is always continuously differentiable \cite{bauschke2011convex, combettes2005signal} even if $U(x)$ is not, the authors of \cite{durmus2018efficient} replace $ \nabla \log p(x)  $  by $ - \nabla U^{(\lambda)}(x) =  \left( \operatorname{prox}_{\lambda, U}(x) - x  \right) / \lambda  $, resulting in Moreau-Yoshida regularized ULA (referred to as MYULA), which requires the proximal operator of $U(x)$ in each iteration of \eqref{ULA}.

In a more general case where the prior $p(x)$ is not available in closed form, the authors of \cite{laumont2022bayesian} propose a plug-and-play (PnP) denoising-based approach for learning the prior \cite{venkatakrishnan2013plug, ryu2019plug}. This is achieved by training a Lipschitz-continuous Gaussian denoiser $D_\varepsilon(x)$. More precisely, $D_{\varepsilon}(x)$ is trained on a given dataset $\{ x_n \}_{n = 1}^{N}$ by learning to remove Gaussian noise of zero-mean and $\varepsilon$ variance added to the clean images $x_n$, which are  i.i.d. samples of $p(x)$.  The ideal minimum mean-squared-error (MMSE) denoiser takes the form
\begin{align}
D_{\varepsilon}(x) 
 =(2 \pi \varepsilon)^{-d / 2} \int_{\mathbb{R}^d} \tilde{x} \exp \left[-\|x-\tilde{x}\|^2 /(2 \varepsilon)\right] p(\tilde{x}) \mathrm{d} \tilde{x}.
\end{align}
The noisy data follows the Gaussian-smoothed prior 
\begin{equation*}
\label{Gaussian_smoothing_convolution}
    p_\varepsilon(x) =(2 \pi \varepsilon)^{-d / 2} \int_{\mathbb{R}^d} \exp \left[-\|x-\tilde{x}\|_2^2 /(2 \varepsilon)\right] p(\tilde{x}) \mathrm{d} \tilde{x},
\end{equation*}
which is the convolution of the non-explicit prior $p(x)$ with a Gaussian smoothing kernel. Similar to the Moreau envelope \cite{bauschke2011convex, combettes2005signal}, $p_\varepsilon$ is always differentiable and satisfies Tweedie's identity \cite{efron2011tweedie}: $\varepsilon \nabla \log p_{\varepsilon}(x)=D_{\varepsilon}(x)-x$.
While computing $\nabla \log p(x)$ could be intractable, one can use $\nabla\log p_\varepsilon(x)$ as a surrogate in \eqref{ULA}, leading to the PnP-ULA approach \cite{laumont2022bayesian}:
\begin{equation}
\label{PnP-ULA}
\begin{aligned}
\text{(PnP-ULA)}:~\quad X_{k+1}&= X_k+\delta \nabla \log p\left(y | X_k\right) \\
& \qquad + \dfrac{\delta \alpha }{\varepsilon}\left(D_{\varepsilon}\left(X_k\right)-X_k\right)  + \dfrac{\delta }{\lambda}\left(\Pi_C\left(X_k\right)-X_k\right) +\sqrt{2 \delta} Z_{k+1},
\end{aligned}
\end{equation} 
where $\alpha>0$ is a regularization parameter associated with the PnP prior and $ \{Z_{k}\}_k$ are i.i.d. drawn from $ \mathcal{N}(0, I^d)$. A projection  $\Pi_C\left(X_k\right)$ onto a convex and compact set $C$ is added in each iteration to enable the theoretical analysis for PnP-ULA. $\lambda>0$ is a parameter associated with the operator $\Pi_C - \operatorname{Id}$.  Moreover,  the Lipschitz continuity of the denoiser $D_\varepsilon(x)$ is required for convergence. A detailed convergence analysis of \eqref{PnP-ULA} is available in \cite{laumont2022bayesian}.

\subsection{Normalizing Flow}\label{sec:NF}
Similar to a PnP prior, a flow-based model can also serve as a prior. A flow-based model seeks to express $x\in \mathbb{R}^{d}$ as 
\begin{equation}
\label{normalizing flow}
x=T(z),
\end{equation}
where $T:\mathbb{R}^{d}\rightarrow \mathbb{R}^{d}$ is an invertible transformation applied to $z\in \mathbb{R}^{d}$, where $z \sim q_{z}(z)$. Here, $q_{{z}}(z)$ is the input (or, latent) distribution of the flow-based model and is generally chosen to be a distribution that can be sampled easily, such as a multivariate Gaussian\cite{rezende2015variational, kingma2016improved, papamakarios2021normalizing, kobyzev2020normalizing}.   Apart from $T:\mathbb{R}^d\rightarrow \mathbb{R}^d$ being invertible, both $T$ and $T^{-1}$ must be differentiable \cite{rezende2015variational, papamakarios2021normalizing}. The flow-based model is also called \textit{normalizing flow} since $T^{-1}$  implicitly transforms $q(x)$, the distribution of $x$, into a normal distribution.  In practice, $T$ is typically implemented with an invertible neural network \cite{ dinh2016density, kingma2018glow}. By a change of variables in \eqref{normalizing flow}, the distribution of $x$ can be written as
\begin{equation}
\label{NF:change of variable}
\begin{aligned}
    q(x)&=q_{z}(z)\left|\operatorname{det} J_{T}(z)\right|^{-1} =q_{z}\left(T^{-1}(x)\right)\left|\operatorname{det} J_{T^{-1}}(x)\right|,
\end{aligned}
\end{equation}
 where $ z=T^{-1}(x)$ and $J_{T}(z)$ is the $d \times d$ Jacobian matrix of $T$.
Many normalizing flows \cite{papamakarios2021normalizing, kingma2018glow, rezende2015variational, kingma2016improved, papamakarios2017masked}  use specific  network architectures such that $T^{-1}$ is a triangular mapping, that is, the Jacobian $ J_{T^{-1}}(x) $ is a triangular matrix, which simplifies the calculation of $ \left|\operatorname{det} J_{T^{-1}}(x) \right|$.  Note that $T$ is used to generate $x$ from $z$, and $T^{-1}$ is needed for evaluating the density $q(x)$.

Some works on normalizing flow use coupling layers in the network to make  $T^{-1}$  a triangular mapping \cite{dinh2014nice, dinh2016density, papamakarios2017masked, kingma2016improved, kingma2018glow}. Denote $G(x) = T^{-1}(x)$, $G:\mathbb{R}^{d}\rightarrow \mathbb{R}^{d}$. Let $x_j$ be the $j$-th element of $x$ and $x_{<j}$ be the elements before $x_j$, i.e. $x_1, \cdots, x_{j - 1}$. Then, for one-layer network, \cite{jaini2020tails} summarizes the  coupling layer-based flows as $G_j(x_j, x_{<j}) = \varphi_j(x_{<j}) x_j + \eta_j(x_{<j}) $, where $G_j$ is the $j$-th element of the vector $G(x)$ and the functions $\phi_j$ and $\eta_j$ map $x_{<j}$ to a real number. The Jacobian $J_G(x)$ is triangular since $G_j$ only depends on $x_j$ and $x_{<j}$.

Assume that the unknown prior distribution that we aim to learn is $p(x)$. Then, the forward KL divergence between the target distribution $p(x)$ and the output distribution $q(x)$ of the NF model \cite{rezende2015variational, papamakarios2021normalizing, kobyzev2020normalizing} can be written as  
\begin{align}
\label{KL-NF}
 D_{\text{KL}}\left( p   , q \right) 
&=-\mathbb{E}_{p ({x})}\left[\log q({x}  )\right]+\text {const.} \\
& =-\mathbb{E}_{p ({x})}\left[ \log q_{{z}}\left(T^{-1}({x}  )  \right)+\log \left|\operatorname{det} J_{T^{-1}}({x}  )\right|\right]+\text {const.} \nonumber
\end{align}
 When the transformation $T$ is parameterized by an invertible neural network $ T_\theta$ with parameters $\theta\in\Theta$, we denote the parameterized density of $x$ as $q_\theta(x)$ and the optimization problem of learning $T_\theta$ reads:
\begin{equation}
\label{KL-NF-parameterized}
    \underset{\theta\in \Theta}{\min }~ D_{\text{KL}}(p,q_\theta).
\end{equation}
Given samples $\left\{{x}_n\right\}_{n=1}^N$ drawn i.i.d. from $p ({x})$, we can estimate the expectation in \eqref{KL-NF} by Monte Carlo averaging over the training samples $\{ x_n \}_{n = 1}^N$. Correspondingly, the loss function for training the NF model becomes 
\begin{align}
\label{KL-real_loss}
& \mathcal{L}(\theta) =  - \dfrac{1}{N}\sum_{i = 1}^{N} \left(   \log q_{{z}}\left(T_\theta^{-1}({x_i}  )  \right)+\log \left|\operatorname{det} J_{T_\theta^{-1}}({x_i}  )\right|   \right)+ \text{const.}
\end{align}
Generally, it is reasonable to assume that the data samples $\{ x_i \}_i^N$ lie within a compact set $C_R \subset \mathbb{R}^d$. In particular, when the flow-based model is learned on imaging data, it is common to set $C_R = [0, 1]^d$. Knowing the set where the data samples lie will give us the intuition to select some parameters in the next section. From the numerical observations,  the networks also partially  know $C_R$ while trained from the data  - the knowledge of $C_R$ is implicitly encapsulated in a well-trained flow model, meaning that most generated samples using a well-trained NF model fall within $C_R$.

\subsection{ULA with NF-prior }
In this section, we propose a framework for sampling using the LMC algorithm based on a pre-trained normalizing flow network. Given data samples $\{x_n \}_{n = 1}^N$  drawn   i.i.d.  from $p(x)$, one can approximate $p(x)$ by learning a flow-based model $x = T_\theta(z)$, with output distribution $q_\theta(x) = q_{z}\left(T_\theta^{-1}(x)\right)\left|\operatorname{det} J_{T_\theta^{-1}}(x)\right|$. Once $q_\theta(x)$ is learned, $\log q_\theta(x)$ is always differentiable since $T_{\theta}$ and $T_{\theta}^{-1}$ are differentiable.  By replacing $p(x)$ with $q_\theta(x)$ in \eqref{ULA}, the ULA scheme boils down to
\begin{equation*}
    X_{k+1}=X_k+\delta \nabla \log p\left(y | X_k\right)+ \delta \nabla \log q_\theta( X_k ) +\sqrt{2 \delta} Z_{k+1}.
\end{equation*}
Since convexity of $-\log q_\theta(x)$ and the Lipschitz continuity of its gradient are not guaranteed to be satisfied, one does not yet have the sufficient conditions to infer convergence and numerical stability similar to the cases in \cite{dalalyan2017theoretical, durmus2017nonasymptotic}. In this work, we follow \cite{laumont2022bayesian} to impose a projection $\Pi_C\left(X_k\right)$ onto a convex and compact set $C$ to ensure that the posterior distribution is well-defined and propose the resulting NF-ULA algorithm (c.f. Algorithm \ref{NF-ULA}). 
\begin{algorithm}[!htbp]
\caption{Normalizing Flow-based Unadjusted Langevin algorithm (NF-ULA)}\label{NF-ULA}
\begin{algorithmic}
\State Input: $y\in \mathbb{R}^m$, $X_0\in \mathbb{R}^d$, $\alpha > 0$, $\lambda > 0$, $K\in \mathbb{N}$, $C\subset \mathbb{R}^d$
\State $\mathrm{L}_y$: Lipschitz constant of $\nabla \log p(y| x)$.
\State $\mathrm{L}$: Lipschitz constant of  $\nabla \log q_\theta(x)$. 
\State Output: $\{ X_k \}_{k = 1}^{K}$
\\
\State Set: $k = 0$, $\delta < (1 / 6)\left(\mathrm{L}_y+ \alpha\mathrm{L} +1 / \lambda\right)^{-1} $. 
\State Initialize $X_0$ according to the considered problems.
\While{$k < K$}
\State $Z_{k+1} \sim \mathcal{N}(0, I^d)$
\State $ X_{k+1}=X_k+\delta  \nabla \log p\left(y | X_k\right) + \delta \alpha \nabla \log q_\theta( X_k ) + \dfrac{\delta}{\lambda}\left(\Pi_C\left(X_k\right)-X_k\right) +\sqrt{2 \delta} Z_{k+1} $
\State $k = k + 1$
\EndWhile
\end{algorithmic}
\end{algorithm}
The parameter $\alpha > 0$ controls how strongly the regularization of $q_\theta$ is imposed and $\lambda$ controls the amount of the projection $(\Pi_C  - \mathrm{Id})$ enforced. Theoretical analysis of NF-ULA is presented in Sec. \ref{sec:theoretical_analysis}, while in Sec. \ref{sec:Experiments}, we provide some general guidelines for selecting the hyper-parameters involved in NF-ULA. One can efficiently compute $\nabla \log q_\theta(x)$ using the automatic differentiation libraries in the standard deep learning frameworks (such as \texttt{PyTorch}). 

\noindent \textbf{Remark}: Algorithm \ref{NF-ULA} only requires evaluating  the   $\nabla \log q_\theta(x)$    and its Lipschitz constant. Our theoretical analysis in Sec. \ref{sec:theoretical_analysis} depends on the properties of $q_{\theta}(x)$ and holds even when $q_{\theta}$ does not arise from a normalizing flow. This is essential since in our CT experiments in Sec. \ref{sec:CT}, we utilize \textit{patchNR} \cite{altekruger2022patchnr}, a normalizing flow-based regularizer which cannot generate $x$ by \eqref{normalizing flow} but is able to evaluate the log gradient $\nabla \log q_\theta(x)$. Moreover, since $q_\theta(x)$ can also be evaluated given $x$, Algorithm \ref{NF-ULA} can be extended to a Metropolis-adjusted version by adding an accept-reject step.     We leave this as a possible future work.


It is imperative to understand why the projection $(\Pi_C - \mathrm{Id}) $ is necessary for the convergence of NF-ULA. Let $\iota^{(\lambda)}_C (x)$ be the $\lambda$-Moreau envelope \cite{bauschke2011convex} of the indicator function 
\begin{equation*}
\iota_C(x) = 
\begin{cases}
0,       & x\in C, \\
+\infty, & x\notin C.
\end{cases}
\end{equation*}
Then, we have that
\begin{equation*}
\begin{aligned}
& \iota^{(\lambda)}_C(x) := \underset{u\in \mathbb{R}^d}{\inf}\, \left(\iota_C(u)+\frac{1}{2\lambda}\|x-u\|_2^2\right) = \dfrac{1}{2\lambda}\left\| x - \Pi_C(x)  \right\|_2^2,
\\
&\text{and\,\,} \nabla \iota^{(\lambda)}_C(x) = \dfrac{x - \operatorname{Prox}_{\iota_C}(x)}{\lambda} = \dfrac{x - \Pi_C(x) }{\lambda},
\end{aligned}
\end{equation*}
where $\Pi_C$ is the projection operator on the convex and compact (i.e., closed and bounded) set $C\subset \mathbb{R}^d$. Define $p_\lambda(x | y )$ as
\begin{equation}
\label{definition:posterior_true}
    p_\lambda(x | y ) = \dfrac{
    p(y | x) q^\alpha_\theta(x) \exp(-\iota^{(\lambda)}_C(x))   }
    {\int_{\mathbb{R}^d} p(y | \tilde{x}) q^\alpha_\theta(\tilde{x}) \exp(-\iota^{(\lambda)}_C(\tilde{x})) \mathrm{d} \tilde{x} },
\end{equation}
where the exponent $\alpha>0$. The subscript $\lambda$ in $p_\lambda$ underlines the distinction from the posterior $p(x|y) = p(y|x)p(x) / p(y)$. Since $\theta$ is fixed if the NF is pre-trained and $\alpha$ is adjusted in the experiments section, they are not in the notation of $p_\lambda$ for brevity.   We show in Sec. \ref{sec:well-posedness} that $p_\lambda(x | y)$ is well-defined and therefore the projection term is necessary for NF-ULA, without which, \eqref{definition:posterior_true} is not guaranteed to be well-defined in our settings.   Denote by $\pi_{\lambda, y}$ (which we will write as $\pi_{\lambda}$ for brevity) the probability measure whose density is $p_\lambda(x | y)$ in \eqref{definition:posterior_true}, i.e.,
\begin{equation}
\label{definition:posterior_measure}
\dfrac{\mathrm{d} \pi_{\lambda}}{\mathrm{d}\pi_{\text{leb}}} (x) = p_\lambda(x | y),
\end{equation}
where $\pi_{\text{leb}}$ denotes the Lebesgue measure. Then, NF-ULA in Algorithm \ref{NF-ULA} is essentially equivalent to
\begin{equation}
\label{NF-ULA-simplified}
    X_{k+1}=X_k+\delta \nabla \log  p_\lambda(X_k | y) + \sqrt{2 \delta} Z_{k+1}.
\end{equation}
For standard ULA \eqref{ULA}, the tail-decay condition (log distribution tail asymptotically proportional to minus quadratic function) was first studied in \cite{roberts1996exponential, stramer1999langevin} and was shown to imply the convergence of ULA. For NF-ULA \eqref{NF-ULA-simplified}, we want to emphasize that in most of our experiments, NF-ULA is convergent while using a well-pre-trained normalizing flow, even without the projection term. This is presumably because the density $q_\theta$ of a well-trained normalizing flow already satisfies the tail-decay condition \cite{roberts1996exponential, stramer1999langevin} and most of the probability mass lies within $C$. For the cases where the normalizing flow is poorly trained, one should select a smaller $C$, without which the samples generated by NF-ULA will go far beyond our expected region (for imaging it is $C_R = [0, 1]^d$).

\section{Theoretical Analysis}\label{sec:theoretical_analysis}
We define some useful notations for our analysis in Sec. \ref{sec:notations} and present a theoretical analysis (well-definedness and well-posedness)  of the Bayesian posterior $p_\lambda(x| y)$ in Sec. \ref{sec:well-posedness}. Subsequently, we prove the convergence and non-asymptotic bias of NF-ULA in Sec. \ref{sec:nonasymptotic-NF-ULA}. 

\subsection{Notations}\label{sec:notations}
Denote by $\mathcal{B}\left(\mathbb{R}^d\right)$ the Borel $\sigma$-field of $\mathbb{R}^d$. Let $\mu$ be a probability measure on $\left(\mathbb{R}^d, \mathcal{B}\left(\mathbb{R}^d\right)\right)$ and $f$ be a $\mu$-integrable function. Denote by $\mu(f)$ the integral of $f$ w.r.t. $\mu$. For measurable $f: \mathbb{R}^d \rightarrow \mathbb{R}$  and measurable $V$ : $\mathbb{R}^d \rightarrow[1, \infty)$ , the $V$-norm of $f$ is defined as $\|f\|_V=\sup _{\tilde{x} \in \mathbb{R}^d}|f(\tilde{x})| / V(\tilde{x})$. Let $\xi$ be a finite signed measure on $\left(\mathbb{R}^d, \mathcal{B}\left(\mathbb{R}^d\right)\right)$. Then the $V$-total variation norm  of $\xi$ is defined as
\begin{equation}
\label{definition:V-norm}
\|\xi\|_V=\sup _{\|f\|_V \leqslant 1}\left|\int_{\mathbb{R}^d} f(\tilde{x}) \mathrm{d} \xi(\tilde{x})\right| .
\end{equation}
Note that if $V=1$, then $\|\cdot\|_V$ is the total variation $\|\cdot\|_{\text{TV}}$. $\|\cdot\|_V$ is weaker than $\|\cdot\|_{\text{TV}}$ and from the definitions one has $\|\xi\|_{\text{TV}}  \leqslant\|\xi\|_V$. $\|\cdot\|_V$ has been used a lot in the studies of ULA\cite{durmus2017nonasymptotic, de2019convergence, laumont2022bayesian}.

We denote by $\mathscr{P}\left(\mathbb{R}^d\right)$ the set of probability measures over $\left(\mathbb{R}^d, \mathcal{B}\left(\mathbb{R}^d\right)\right)$ and for any $m \in \mathbb{N}, \mathscr{P}_m\left(\mathbb{R}^d\right)=\left\{\nu \in \mathscr{P}\left(\mathbb{R}^d\right): \int_{\mathbb{R}^d}\|\tilde{x}\|^m \mathrm{~d} \nu(\tilde{x})<+\infty\right\}$. Denote by $\mathbf{W}_p$ as Wasserstein-$p$ metric:
\begin{equation}
\mathbf{W}_p(\mu, \nu) = \left( \inf_{\gamma \in \Gamma(\mu, \nu)} \mathbf{E}_{(x, y)\sim \gamma}\| x - y \|^p \right)^{1/p},\quad p\geqslant 1,
\end{equation}
where $\Gamma(\mu, \nu)$ is the set of all joint probability whose marginal distributions are $\mu$ and $\nu$ respectively.

Let $b \in \mathrm{C}\left(\mathbb{R}^d, \mathbb{R}^d\right)$ where  $C\left(\mathbb{R}^d, \mathbb{R}^d\right)$ stands for the set of all continuous functions from $\mathbb{R}^d$ to $\mathbb{R}^d$. We consider the Markov chain $\left(X_k\right)_{k \in \mathbb{N}}$ given by the following recursion for any $k \in \mathbb{N}$ and $x \in \mathbb{R}^d$, initialized at $X_0=x$:
\begin{equation*}
\begin{aligned}
    & X_{k+1}=X_k+\gamma b\left(X_k\right)+\sqrt{2 \gamma} Z_k,
\end{aligned}
\end{equation*}
where $ \gamma>0$ and $\left\{Z_k: k \in \mathbb{N}\right\}$ a family of i.i.d. Gaussian random variables with zero mean and identity covariance matrix. We define its associated Markov kernel $\mathrm{R}_\gamma$ : $\mathbb{R}^d \times \mathcal{B}\left(\mathbb{R}^d\right) \rightarrow[0,1]$ as follows for any $x \in \mathbb{R}^d$ and $\mathrm{A} \in \mathcal{B}\left(\mathbb{R}^d\right)$: 
\begin{equation*}
    \mathrm{R}_\gamma(x, \mathrm{~A})=(2 \pi)^{-d / 2}\int_{\mathbb{R}^d} \mathbf{1}_{\mathrm{A}}(x+\gamma b(x)+\sqrt{2 \gamma} z) \exp \left[-\|z\|^2 / 2\right] \mathrm{d} z ,
\end{equation*}
where $\mathbf{1}_\mathrm{A}(x)  $ is the function taking the value $1$ if $x\in \mathrm{A}$ or $0$ if $x\notin \mathrm{A}$.  We say that $\mathrm{R}_\gamma$ satisfies a discrete drift condition $\mathbf{D}_{\mathrm{d}}(W, \zeta_\mathrm{d}, c)$ if there exist $\zeta_\mathrm{d} \in[0,1), c \geqslant 0$ and a measurable function $W: \mathbb{R}^d \rightarrow[1,+\infty)$ such that for all $x \in \mathbb{R}^d$
\begin{equation*}
    \mathrm{R}_\gamma W(x) \leqslant \zeta_\mathrm{d} W(x)+c,
\end{equation*}

where $\mathrm{R}_\gamma W(x) := \int_{\mathbb{R}^d} \mathrm{R}_\gamma(x, \mathrm{d} \tilde{x})W(\tilde{x}) $.  Note that this drift condition implies the existence of an invariant probability measure if $\mathrm{R}_\gamma$ is a Feller kernel and the level sets of $W$ are compact, see \cite{de2019convergence} and Theorem 12.3.3 in \cite{douc2018markov}.  

Similarly, let $b \in \mathrm{C}\left(\mathbb{R}^d, \mathbb{R}^d\right)$ such that for any $x \in \mathbb{R}^d$, the following SDE admits a unique strong solution
\begin{equation}
\label{eq:SDE}
\begin{aligned}
    & \mathrm{d} \mathbf{X}_t=b\left(\mathbf{X}_t\right) \mathrm{d} t+\sqrt{2} \mathrm{~d} \mathbf{B}_t,
    \\
    & \mathbf{X}_0=x,
\end{aligned}
\end{equation}
where  $\left(\mathbf{B}_t\right)_{t \geqslant 0}$ is a $d$-dimensional Brownian motion. For any $x \in \mathbb{R}^d$ and $\mathrm{A} \in \mathcal{B}\left(\mathbb{R}^d\right)$, equation \eqref{eq:SDE} defines a Markov semi-group $\left(\mathrm{P}_t\right)_{t \geqslant 0}$  by $\mathrm{P}_t(x, \mathrm{~A})=\mathbb{P}\left(\mathbf{X}_t \in \mathrm{A}\right)$ where $\left(\mathbf{X}_t\right)_{t \geqslant 0}$ is the solution of \eqref{eq:SDE} with $\mathbf{X}_0=x$. For any $f \in \mathrm{C}^2\left(\mathbb{R}^d, \mathbb{R}\right)$, define the generator $\mathcal{A}$ of $\left(\mathrm{P}_t\right)_{t \geqslant 0}$ by $\mathcal{A} f=\langle\nabla f, b(x)\rangle+\Delta f$, where $\Delta$ is the Laplace operator. We say that $\left(\mathrm{P}_t\right)_{t \geqslant 0}$ on $\mathbb{R}^d \times \mathcal{B}\left(\mathbb{R}^d\right)$ with extended infinitesimal generator $(\mathcal{A}, \mathrm{D}(\mathcal{A}))$ (see e.g. \cite{meyn1993stability} for the definition of $(\mathcal{A}, \mathrm{D}(\mathcal{A}))$ ) satisfies a continuous drift condition $\mathbf{D}_{\mathrm{c}}(W, \zeta, \beta)$ if there exist $\zeta>0, \beta \geqslant 0$ and a measurable function $W: \mathbb{R}^d \rightarrow[1,+\infty)$ with $W \in \mathrm{D}(\mathcal{A})$ such that for all $x \in \mathbb{R}^d$,
\begin{equation*}
    \mathcal{A} W(x) \leqslant-\zeta W(x)+\beta.
\end{equation*}
This assumption is the continuous counterpart of the discrete drift condition $\mathbf{D}_{\mathrm{d}}(W, \zeta_\mathrm{d}, c)$, which will be used in Appendix \ref{proof:nonasymptotic-bias}.

\subsection{Well-posedness of the Bayesian solution}\label{sec:well-posedness}
In this section, we first prove that the posterior distribution \eqref{definition:posterior_true} is well-defined. Secondly, we prove the well-posedness for the Bayesian solution, i.e., the Lipschitz continuity of the posterior measure \eqref{definition:posterior_measure} with respect to changes in $y$. To start with, we give a lemma that will be used later.
\begin{lemma}\label{lemma:finite-moment}
Let $\lambda>0$. For any convex and compact subset $C$ of $\mathbb{R}^d$ and for all $k\in \mathbb{N}$, it holds that
\begin{equation*}
\int_{\mathbb{R}^d} \left\| x \right\|^k  \exp\left( - \dfrac{\left\| x - \Pi_C(x) \right\|_2^2}{2\lambda}\right) \mathrm{d} x < +\infty.
\end{equation*}
\end{lemma}

\begin{proof}
See Appendix \ref{proof:finite-moment}.  
\end{proof}

Lemma \ref{lemma:finite-moment} implies that the integral of any polynomials multiplied by $\exp\left( - \iota^{(\lambda)}_C \right)$, where $\iota^{(\lambda)}_C=\dfrac{\left\| x - \Pi_C(x) \right\|_2^2}{2\lambda}$, is finite. To prove that $ p_\lambda(x| y) $ and $\pi_\lambda$ are well-defined,  besides Lemma \ref{lemma:finite-moment}, we need an assumption about the boundedness of the prior and the likelihood.
\begin{assumption}\label{Assumption1}
The distribution learned by NF is bounded, i.e., $\underset{x \in \mathbb{R}^d}{\sup} q_\theta(x)<+\infty $. Moreover, for any $y \in \mathbb{R}^m$, $ \underset{x \in \mathbb{R}^d}{\sup} p(y | x)<+\infty$ and $p(y | \cdot) \in \mathrm{C}^1\left(\mathbb{R}^d,(0,+\infty)\right)$.
\end{assumption}

Since $q_\theta(x)$ is a distribution induced by normalizing flow and $q_\theta(x)$ is continuous on $\mathbb{R}^d$, intuitively $  \underset{x \in \mathbb{R}^d}{\sup} q_\theta(x)$ is bounded and Assumption \ref{Assumption1} is easily satisfied.   To give rigorous proof,  we state the following proposition which  assumes a similar triangular network architecture as mentioned in Sec. \ref{sec:NF}.  
\begin{proposition}\label{prop:bounded_q}
Assume that the input distribution $q_z(z)$ to the normalizing flow network is the standard normal distribution. Assume that  $ T^{-1}(x) = G^{(k)}\circ \cdots \circ G^{(1)}(x)   $ is a composition of $k$  coupling layers and each of the layer $G^{(i)}: x^{(i)} \mapsto x^{(i+1)} $, $G^{(i)}:\mathbb{R}^{d}\mapsto \mathbb{R}^{d}$ is given by
\begin{equation}
    G^{(i)}_j(x^{(i)}_j, x^{(i)}_{<j}) = \varphi^{(i)}_j(x^{(i)}_{<j}) x^{(i)}_j + \eta^{(i)}_j(x^{(i)}_{<j}), \,\,j = 1,\cdots, d.
    \label{eq:struct_of_inv_map1}
\end{equation}
Denote $x^{(1)} = x$ and $x^{(k+1)} = z$. 
If  $\varphi^{(i)}_j$s are bounded, then $ \log q_\theta(x)$ is upper bounded on $\mathbb{R}^d$. 
\end{proposition}
\begin{proof}
See Appendix \ref{proof:bounded_q}.
\end{proof} 
Using Lemma \ref{lemma:finite-moment}, we can then prove that the normalizing constant in the expression for $p_\lambda(x | y )$ in \eqref{definition:posterior_true} is finite.
\begin{corollary}\label{cor:well-defined}
Suppose Assumption \ref{Assumption1} holds. Let $\lambda>0$. Then, for any convex and compact set $C$ and $\alpha>0$, we have
\begin{equation*}
\int_{\mathbb{R}^d} p(y | x) q^\alpha_\theta( x )  \exp\left( - \dfrac{\left\| x - \Pi_C(x) \right\|_2^2}{2\lambda}\right) \mathrm{d} x < +\infty.
\end{equation*}
Hence, $p_\lambda(x| y)$ in \eqref{definition:posterior_true} is well-defined.  
\end{corollary}
\begin{proof}
Letting $k = 0$ in Lemma \ref{lemma:finite-moment} and using Assumption \ref{Assumption1}, we conclude the proof.
\end{proof}

\noindent \textbf{Remark}: Although $\int_{\mathbb{R}^d} q_\theta(x) \mathrm{d} x = 1$,  $\int_{\mathbb{R}^d} q_\theta^\alpha(x) \mathrm{d} x $ may not be finite in rare cases. This depends on how heavy the tail of $q_\theta(x)$ is. Corollary \ref{cor:well-defined} shows that multiplying $q_\theta^\alpha(x)$ with $\exp\left(-\iota^{(\lambda)}_C(x)\right)$ always leads to a finite integral, regardless of the tail behavior of $q_\theta(x)$. 

Now, we establish the well-posedness of the posterior measure $\pi_\lambda$ in the following proposition. Note that the local Lipschitz stability of posterior distribution in the observation has been studied in \cite{sprungk2020local, latz2020well} and applied to posterior sampling with PnP prior\cite{laumont2022bayesian} and generative models in \cite{altekruger2023conditional}. Apart from the considered $\iota^{(\lambda)}_C(\tilde{x})$, Proposition  \ref{prop:well-posedness} and Proposition SM5.3 in \cite{SupplementaryMaterials} are based on similar ideas.

\begin{proposition}
\label{prop:well-posedness}
Suppose Assumption \ref{Assumption1} holds and that there exist continuous functions $\Phi_1: \mathbb{R}^d \rightarrow[0,+\infty)$ and $\Phi_2: \mathbb{R}^m \rightarrow[0,+\infty)$ such that for any $x \in \mathbb{R}^d$ and $y_1, y_2 \in \mathbb{R}^m$, the following are satisfied:
\begin{align*}
&\big|\log \left(p\left(y_1 | x\right)\right)-\log \left(p\left(y_2 | x\right)\right)\big|  \leqslant\left(\Phi_1(x)+\Phi_2\left(y_1\right)+\Phi_2\left(y_2\right)\right)\left\|y_1-y_2\right\|,\\
&\text{and\,\,}
 \int_{\mathbb{R}^d}\left(1+\Phi_1(\tilde{x})\right) \exp \left[c_0 \,\Phi_1(\tilde{x})   - \iota^{(\lambda)}_C(\tilde{x})   \right]  \mathrm{d} \tilde{x} < +\infty,
\end{align*}
for all $c_0 > 0$.  
Then, $y \mapsto \pi_{\lambda, y} $ defined in \eqref{definition:posterior_measure}) is locally Lipschitz w.r.t. the total-variation (TV) norm $\|\cdot\|_\mathrm{TV}$, i.e., for any compact set ${K}$, there exists $M_{{K}} \geqslant 0$ such that for any $y_1, y_2 \in {K},\left\|   \pi_{\lambda, y_1} - \pi_{\lambda, y_2} \right\|_{\mathrm{TV}} \leqslant M_{{K}}\left\|y_1-y_2\right\|$.
\end{proposition}
\begin{proof}
See Appendix \ref{proof:well-posedness}.
\end{proof}
For Gaussian likelihood $p(y|x)$, the conditions in Proposition \ref{prop:well-posedness} are satisfied when $\Phi_1(x) = c_1 \| x \|_2 $ and $\Phi_2(y) = c_2\| y \|_2 $ with positive constants $c_1$ and $c_2$.

\subsection{Convergence of NF-ULA}\label{sec:nonasymptotic-NF-ULA}
Most of the existing works on ULA for non-convex potentials \cite{durmus2017nonasymptotic, majka2020nonasymptotic, cheng2018sharp, erdogdu2021convergence, laumont2022bayesian} assume Lipschitz-continuity of the log gradient of the target density.   If the drift term  $ \nabla \log p_\lambda(x| y) $ is not Lipschitz, from \cite{ikeda2014stochastic, karatzas1991brownian}, it cannot generally be guaranteed that the SDE \eqref{eq:LSDE_posterior} will have a unique strong solution. This is why one must investigate the Lipschitz continuity of $ \nabla \log p_\lambda(x| y) $ before studying the convergence of NF-ULA. First, we make an assumption about the Lipschitz-continuity of $\nabla \log (p(y| \cdot))$:  




\begin{assumption}\label{Assumption2}
$\nabla \log (p(y | x))$ is $\mathrm{L}_y$-Lipschitz continuous in $x$, where $\mathrm{L}_y>0$ is a constant.
\end{assumption}

Note that Assumption \ref{Assumption2} is generally satisfied for common imaging inverse problems. One example is the popular Gaussian likelihood where $ p(y| x) \propto \exp\left(- {\left\| y - Ax \right\|_2^2} / (2\sigma^2)  \right) $, for which $\mathrm{L}_y= \| A^\top A  \| / \sigma^2 $. 
\begin{lemma}
\label{lemma:Lipschitz}
Under Assumption \ref{Assumption2}, $ \nabla \log p_{\lambda}(x | y)$ is Lipschitz continuous if and only if $ \nabla \log q_\theta(x) $ is Lipschitz continuous.
\end{lemma}
\begin{proof}
See Appendix \ref{proof:Lipschitz}.
\end{proof}
For convenience, we explicitly define the Lipschitz condition on the log gradient of $q_{\theta}(x)$ in the following assumption:
\begin{assumption}
\label{Assumption3}
There exist $\mathrm{L} \geqslant 0$ such that for any $  x_1, x_2 \in \mathbb{R}^d$, 
\begin{equation*}
\left\| \nabla \log q_\theta \left(x_1\right)-  \nabla \log q_\theta\left(x_2\right)\right\| \leqslant \mathrm{L}\left\|x_1-x_2\right\|.
\end{equation*}
\end{assumption}
It is therefore natural to ask how to enforce Assumption \ref{Assumption3} on the NF-based image prior $q_{\theta}(x)$ during training or by the network architecture. There have been some studies about the Lipschitz continuity of the invertible transform $T_{\theta}$ \cite{kobyzev2020normalizing, papamakarios2021normalizing, verine2021expressivity}, the Lipschitz constants of invertible neural networks by changing the latent distribution from a standard normal one to a Gaussian mixture model \cite{hagemann2021stabilizing},   the Lipschitz constants of other ``push-forward'' generative models   \cite{salmona2022can}. However,  to the best of our knowledge, there is no study about the Lipschitz continuity of $\nabla \log q_\theta(x) $ until now.  

While the equivalent conditions on $T_{\theta}$ for Assumption \ref{Assumption3} remain unknown, a sufficient condition on $T_{\theta}$ for Assumption \ref{Assumption3} can be obtained easily. For instance, when $T_{\theta}$ is a linear transform mapping a Gaussian distribution $q_z(z)$ to another Gaussian distribution $q_\theta(x)$, Assumption \ref{Assumption3} holds.  However, this may not be true if $T_{\theta}$ is nonlinear. 

As we have mentioned that Assumption \ref{Assumption3} is necessary for the convergence of NF-ULA, we derive a sufficient condition on $T_{\theta}$ for Assumption \ref{Assumption3} to hold. Intuitively, distributions with similar tail behaviors as Gaussian may have similar log gradients as Gaussian, if more conditions are satisfied.  We thus refer to some studies on the tails of normalizing flow priors \cite{jaini2020tails }. Theorem 4 in \cite{jaini2020tails} shows that affine coupling layer-based flows (e.g., NICE \cite{dinh2014nice}, Real-NVP\cite{dinh2016density}, MAF\cite{papamakarios2017masked}, IAF\cite{kingma2016improved}, and Glow \cite{kingma2018glow}) can only map the base normal distribution $q_z(z)$ to a light-tailed distribution $q_\theta(x)$. To be more specific, denote $G(x) = T^{-1}(x)$, where $G(x)$ is a triangular mapping and the Jacobian $J_G(x)$ is a triangular matrix function.  From \cite{jaini2020tails }, generally one can assume that for affine coupling layer-based flows, $G_j(x_j, x_{<j}) = \varphi_j(x_{<j}) x_j + \eta_j(x_{<j}) $, where $G_j$ is the $j$-th element of the vector $G(x)$ and $x_{<j}$ indicate $x_1, \cdots, x_{j - 1}$.  The condition they assume is heuristic: if $\varphi_j$ is bounded above and $\eta_j$ is Lipschitz, then $q_\theta(x)$ is light-tailed. In Glow, \cite{kingma2018glow} these conditions on $\varphi$ and $\eta$ are satisfied and even stricter. Therefore, we are able to prove the Lipschitz continuity of $\nabla \log q_\theta(x)$ in the proposition below, by enforcing a stricter condition on $\varphi$ and $\eta$.
\begin{proposition}\label{prop:Lipschitz}
Assume that the input distribution $q_z(z)$ to the normalizing flow network is the standard normal distribution, and that  $ T^{-1}(x) = G^{(k)}\circ \cdots \circ G^{(1)}(x)   $ is a composition of $k$  coupling layers, where each of the layers $G^{(i)}: x^{(i)} \rightarrow x^{(i+1)} $, $G^{(i)}:\mathbb{R}^{d}\rightarrow \mathbb{R}^{d}$ is given by
\begin{equation}
    G^{(i)}_j(x^{(i)}_j, x^{(i)}_{<j}) = \varphi^{(i)}_j(x^{(i)}_{<j}) x^{(i)}_j + \eta^{(i)}_j(x^{(i)}_{<j}), \,\,j = 1,\cdots, d.
    \label{eq:struct_of_inv_map}
\end{equation}
Denote $x^{(1)} = x$ and $x^{(k+1)} = z$. 
If  $\varphi^{(i)}_j$ is a constant function, $\eta^{(i)}_j$ is Lipschitz and for all $\displaystyle r < j, ~\frac{\partial \eta^{(i)}_j}{\partial x_r}$ is well-defined almost everywhere and piecewise constant on $\mathbb{R}$, then $ \nabla \log q_\theta(x) $ is  Lipschitz continuous on $\mathbb{R}^d$.
\end{proposition}
\begin{proof}
See Appendix \ref{proof:prop_Lipschitz}. 
\end{proof}
The conditions on $\varphi, \eta$ in Proposition \ref{prop:Lipschitz} are satisfied in \textit{Glow} \cite{kingma2018glow} with    \textit{additive  coupling layers}  where  each $\eta$ is a five-layer sequential network with 2D convolutional layers (denoted as \texttt{Conv2d}) and \texttt{ReLU} activations:
\begin{equation*}
\eta(x) = \texttt{Conv2d}(\texttt{ReLu}(\texttt{Conv2d}(\texttt{ReLu}(\texttt{Conv2d}(x))))),
\end{equation*}
where $\texttt{ReLu}(x):= \max(0,x)$ (applied in an element-wise manner) and $\texttt{Conv2d}(x) :=K_{\mathrm{NF}}* x $ denotes a 2D convolution layer acting on $x$ with a kernel $K_{\mathrm{NF}}$. Further, $\varphi = 1$ is used in the additive coupling layer. Note that in Glow, there is an option of using an \textit{affine coupling layer} where $\varphi$ is the sigmoid function $\varphi(x) = 1 / (1 + e^{-x}) $ element-wise. This leads to a more powerful network and can generate better human face images \cite{kingma2018glow}, but $\nabla \log q_\theta(x)$ is not guaranteed to be Lipschitz anymore. This theoretical observation is corroborated by our experiments in Sec. \ref{sec:deblurring}, as we found that NF-ULA with affine coupling layer did not converge. The conditions on $\varphi$ and $\eta$ might be relaxed if $q_z(z)$ is not Gaussian, but this requires re-training the network since most of the popular normalizing flows accept standard Gaussian base distribution as input. We leave these studies on the Lipschitz-continuity of $\nabla \log q_\theta(x)$ for future work. 

In order to prove the convergence of NF-ULA, we need one final assumption.
\begin{assumption}
\label{Assumption4}
There exists $\mathrm{m}_y \in \mathbb{R}$ such that for all $x_1, x_2 \in \mathbb{R}^d$, we have
\begin{align*}
\left\langle\nabla \log p\left(y | x_2\right)-\nabla \log p\left(y | x_1\right), x_2-x_1\right\rangle 
\leqslant-\mathrm{m}_y\left\|x_2-x_1\right\|_2^2 .
\end{align*}
\end{assumption}

This condition is called the \textit{contractivity condition} of $ \nabla \log p(y| x) $ and is used to  prove the contractivity of the drift term $\nabla \log p_\lambda(x|y)$ at infinity (see proofs of Theorem \ref{theorem:contractive} in Appendix \ref{proof:contractive}). Note that the influence of the drift's contractivity condition has been studied in ULA for non-convex potentials \cite{cheng2018sharp, majka2020nonasymptotic, de2019convergence}.

If Assumption \ref{Assumption4} is satisfied with $\mathrm{m}_y>0$, then $x \mapsto  - \log p(y | x)$ is $\mathrm{m}_y$-strongly convex.  If Assumption \ref{Assumption2} is satisfied, then Assumption \ref{Assumption4} holds for $\mathrm{m}_y=-\mathrm{L}_y$. However, we are interested to find $\mathrm{m}_y>-\mathrm{L}_y$ while Assumption \ref{Assumption2} holds, since we will see in the proofs of
Theorem \ref{theorem:contractive} and Theorem \ref{theorem:nonasymptotic-bias} in Appendix \ref{proof:contractive} and \ref{proof:nonasymptotic-bias} that a larger $\mathrm{m}_y$ is beneficial to the convergence of NF-ULA.

In what follows, we introduce the associated stochastic kernel $\mathrm{R}_\delta: \mathbb{R}^d \times \mathcal{B}(\mathbb{R}^d)\rightarrow [0, 1] $ of the NF-ULA (\ref{NF-ULA-simplified}) and the drift $b_\lambda\in \mathrm{C}\left(\mathbb{R}^d, \mathbb{R}^d\right)$:  
\begin{equation}
\label{definition:stochastic-kernel}
\begin{aligned}
 \mathrm{R}_{\delta}(x, \mathrm{~A}) & = (2 \pi)^{-d / 2} \int_{\mathbb{R}^d} \mathbf{1}_{\mathrm{A}}\left(x+\delta b_\lambda(x)+\sqrt{2 \delta} z\right) \exp \left[-\|z\|^2 / 2\right] \mathrm{d} z,
\\
 b_\lambda (x) & = \nabla \log p_\lambda (x | y)  =  \nabla \log p\left(y | x\right)+ \alpha \nabla \log q_\theta( x ) + \dfrac{  \Pi_C\left(x\right)-x  }{\lambda},
\end{aligned}
\end{equation}
where $x\in \mathbb{R}^d $ and $\mathrm{~A}\in \mathcal{B}(\mathbb{R}^d) $.  Here $b_\lambda $ has the subscript $\lambda$ and is different from the $b$  defined in Sec. \ref{sec:notations} because of  $  (\Pi_C\left(x\right)-x)  /\lambda$.  Given $X_k$ in NF-ULA (\ref{NF-ULA-simplified}), $\mathrm{R}_{\delta}(X_k, \cdot)$ is actually a probability measure which defines the transition probability $p(X_{k+1} | X_k )$.

With all the previous four assumptions A \ref{Assumption1}, A \ref{Assumption2}, A \ref{Assumption3}, and A \ref{Assumption4} holding, we can prove that NF-ULA (Algorithm \ref{NF-ULA}) is convergent, or more precisely, the stochastic kernel $\mathrm{R}_\delta$ admits an unique invariant distribution $\pi_{\delta, \lambda}$.   We follow the proof in SM6.2 from \cite{SupplementaryMaterials} but our theorem and proof are slightly different, as we do not include the parameter $\varepsilon$ of PnP denoisers in the condition.  The first thing to prove is that $\mathrm{R}_\delta$ defines a contractive mapping. 


\begin{theorem}
\label{theorem:contractive}
Assume A \ref{Assumption1}, A \ref{Assumption2}, A \ref{Assumption3}, and A \ref{Assumption4}. Assume  $V(x) = 1 + \left\| x \right\|^2, x\in \mathbb{R}^d$. Let $\lambda, \alpha, C, \mathrm{L}_y, \mathrm{L}$ be the ones in NF-ULA (Algorithm \ref{NF-ULA}). Let $\mathrm{m}_y$ be the parameter in A \ref{Assumption4}. Let $\lambda>0  $, such that $2 \lambda\left(\mathrm{L}_y+ \alpha\mathrm{L} -\min (\mathrm{m}_y, 0)\right) \leqslant 1$ and let $\bar{\delta}=(1 / 6)\left(\mathrm{L}_y+ \alpha\mathrm{L} +1 / \lambda\right)^{-1}$. Then for any convex and compact $C$ with $0\in C$, there exist $A_{1 } \geqslant 0$ and $\rho_{1 } \in[0,1)$ such that for any $\delta \in(0, \bar{\delta}], x_1, x_2 \in \mathbb{R}^d$, and $k \in \mathbb{N}$ we have
\begin{equation*}
\begin{aligned}
&\left\|\boldsymbol{\delta}_{x_1} \mathrm{R}_{\delta}^k-\boldsymbol{\delta}_{x_2} \mathrm{R}_{ \delta}^k\right\|_V \leqslant A_{1 } \rho_{1 }^{k \delta}\left(V^2\left(x_1\right)+V^2\left(x_2\right)\right), \text{\,and\,}\\
&\mathbf{W}_1\left(\boldsymbol{\delta}_{x_1} \mathrm{R}_{ \delta}^k, \boldsymbol{\delta}_{x_2} \mathrm{R}_{ \delta}^k\right) \leqslant A_{1 } \rho_{1 }^{k \delta}\left\|x_1-x_2\right\|_2.
\end{aligned}
\end{equation*}
\end{theorem}
\begin{proof}
See Appendix \ref{proof:contractive}.
\end{proof}

In the above theorem the Dirac measures $\boldsymbol{\delta}_{x_1}, \boldsymbol{\delta}_{x_2}$ can be extended to any measures $\nu_1, \nu_2 \in \mathscr{P}_1\left(\mathbb{R}^d\right)$: 
\begin{equation}
\label{generalized-contractive}
\begin{aligned}
\left\|\nu_1 \mathrm{R}_{ \delta}^k-\nu_2 \mathrm{R}_{ \delta}^k\right\|_V 
& \leqslant A_{1   } \rho_{1 }^{k \delta}\left(\int_{\mathbb{R}^d} V^2(\tilde{x}) \mathrm{d} \nu_1(\tilde{x})+\int_{\mathbb{R}^d} V^2(\tilde{x}) \mathrm{d} \nu_2(\tilde{x})\right), 
\\
\mathbf{W}_1\left(\nu_1 \mathrm{R}_{ \delta}^k, \nu_2 \mathrm{R}_{ \delta}^k\right) 
& \leqslant A_{1   } \rho_{1  }^{k \delta}\left(\int_{\mathbb{R}^d}\|\tilde{x}\| \mathrm{d} \nu_1(\tilde{x})+\int_{\mathbb{R}^d}\|\tilde{x}\| \mathrm{d} \nu_2(\tilde{x})\right) .    
\end{aligned}    
\end{equation}
From Theorem 6.18 in \cite{villani2009optimal}, $\left(\mathscr{P}_1\left(\mathbb{R}^d\right), \mathbf{W}_1\right)$ is a complete metric space. For any measure $\nu \in \mathscr{P}_1\left(\mathbb{R}^d\right)$, define $\mathrm{f}: \mathscr{P}_1\left(\mathbb{R}^d\right) \rightarrow \mathscr{P}_1\left(\mathbb{R}^d\right)$ as   $\mathrm{f}(\nu)=\nu \mathrm{R}_{\varepsilon, \delta}$. Then for any $\delta \in(0, \bar{\delta}]$, there exists large enough $\mathrm{m}_\delta \in \mathbb{N}^*$ such that $\mathrm{f}^{\mathrm{m}_\delta}$ is a  contractive mapping. Therefore we can apply the Picard fixed point theorem and we obtain that $\mathrm{R}_{ \delta}$ admits an unique invariant probability measure $\pi_{\delta, \lambda} \in \mathscr{P}_1\left(\mathbb{R}^d\right)$. Since $\pi_{\delta, \lambda}$ is subject to bias comparing with the solution of the SDE $\mathrm{d} X_t=  b_\lambda( X_t )\mathrm{d}t + \sqrt{2} \mathrm{~d} B_t$, in the Theorem below, we follow the proof in SM6.3 from  \cite{SupplementaryMaterials} and give a nonasymptotic bias analysis:

\begin{theorem}\label{theorem:nonasymptotic-bias}
Assume A \ref{Assumption1}, A \ref{Assumption2}, A \ref{Assumption3}, A \ref{Assumption4}. Assume  $V(x) = 1 + \left\| x \right\|^2, x\in \mathbb{R}^d$. Let $\lambda, \alpha, C, \mathrm{L}_y, \mathrm{L}$ be the ones in NF-ULA (Algorithm \ref{NF-ULA}). Let $\mathrm{m}_y$ be the parameter in A \ref{Assumption4}. Let $\lambda>0$ such that $2 \lambda\left(\mathrm{L}_y+ \alpha \mathrm{L} -\min (\mathrm{m}_y, 0)\right) \leqslant 1$ and let $\bar{\delta}=(1 / 6)\left(\mathrm{L}_y+\alpha\mathrm{L} + 1 / \lambda\right)^{-1}$. Then for any $\delta \in(0, \bar{\delta}]$ and $\mathrm{C}$ convex and compact, $ \mathrm{R}_{  \delta  }$ admits an invariant probability measure $\pi_{ \delta, \lambda}$. In addition, there exists $B_1, B_2, B_3 \geqslant 0$, $\tilde{\rho}_1\in [0, 1)$ such that for any $\delta \in(0, \bar{\delta}]$, $k\in \mathbb{N}^*$, 
\begin{equation*}
\left\| \boldsymbol{\delta}_x \mathrm{R}_{\delta}^{k} -  \pi_{\lambda}\right\|_V 
 \leqslant B_1 \tilde{\rho}_{1}^{k\delta} V^2(x) + B_2 V(x)  \sqrt{\delta^2 k \left(d + \dfrac{B_3 \delta}{3}\right)    } . 
\end{equation*}

\end{theorem}
\begin{proof}
See Appendix \ref{proof:nonasymptotic-bias}.    
\end{proof}
\noindent \textbf{Remark}: Note that there is a trade-off of selecting the step-size $\delta$.  In order to achieve a small bias, one needs to set a large time interval $ t =  k\delta $, keep $t$ fixed and use a small step size $\delta$. However, larger $k$ means drawing more samples, resulting in longer computation time. In practice, the burn-in period is incorporated in $t$, in which the Markov Chain is dramatically exploring the state space. 
\section{Experiments in Bayesian Imaging}\label{sec:Experiments}
We apply NF-ULA and PnP-ULA on three inverse problems: image motion deblurring, image inpainting, and limited-angle computed tomography (CT) reconstruction. We compare with PnP-ULA since, to the best of our knowledge, it is the state-of-the-art Langevin algorithm with data-driven non-convex regularizers.

\noindent\textbf{Choice of  $\alpha$}: For both NF-ULA and PnP-ULA on different problems, we fine-tune $\alpha$ such that the peak signal-to-noise ratio (PSNR) of the sample mean gets maximized. While in most cases $\alpha\in (0, 5]$ works well, for NF-ULA it is also related to the architecture of the normalizing flow. For CT reconstruction, we use the pre-trained patchNR, a NF-based regularizer learned on medical images, from the code provided in \cite{altekruger2022patchnr} and choose $\alpha = 5000$. Notably, in the original implementation, the maximum a posteriori estimator was considered, and $\alpha = 700$ was the best choice. 

\noindent\textbf{Choices of  $C$ and $\lambda$}: We only perform the study of choosing different $C$ and $\lambda$ in the deblurring experiments.   From \cite{laumont2022bayesian}, a projection term $ (\text{Id} - \Pi_C) $ is introduced to PnP-ULA  to make sure that the posterior satisfies the tail-decay condition. Therefore, for posterior distributions with a slower tail-decay, a smaller $C$ is recommended.  We found experimentally that NF-ULA was numerically stable when the NF prior was trained for more than 20 epochs, even with a large $C$. In this case, $C$ is chosen to be large enough such that $\Pi_C$ is never activated, since we do not expect to choose a small $C$ to change the behaviors of  NF-ULA if it already converges. For a normalizing flow that is not well trained (less than 5 epochs), it is recommended that $C$ should be the same as the range $C_R$ of the dataset.  In the imaging problems, we have that $ C_R = [0, 1]^d $. See Table \ref{tab:epochs-comparing} for details on the algorithm behaviors of NF-ULA with different choices of $C$ and normalizing flow architectures.   For well-trained normalizing flows in NF-ULA and denoiser in PnP-ULA, we set $C = [-100, 100]^d$. Actually all the samples generated in Tables \ref{tab:deblurring}, \ref{tab:inpainting}, and \ref{tab:CT_Gaussian_noise_limited} never escaped $ [-0.2, 1.2]^d $, indicating that the projection $\Pi_C(x)$ was never activated. We keep $\lambda = 5\times 10^{-5}$, even though different $\lambda$ makes no difference in most of our experiments. 

\noindent\textbf{Choice of $\delta$}:  From the convergence analysis in Theorem \ref{theorem:contractive} and Theorem \ref{theorem:nonasymptotic-bias},   any $\delta < (1 / 6)\left(\mathrm{L}_y+ \alpha\mathrm{L} +1 / \lambda\right)^{-1}$ should work. However, this upper bound is not a strict bound and in practice, it is not easy to know the Lipschitz constant $\mathrm{L}$ of $\nabla \log q_\theta(x)$.  To give an upper bound of $\mathrm{L}$, we calculate the spectral norm  of  $ \nabla^2 \log q_\theta(x)$ through power iteration when $x$ is randomly choosen in $C_R$ and the spectral norm are smaller than $2\times10^5$. This upper bound for $\mathrm{L}$ is still too loose since we find that NF-ULA converges for many $\delta$ larger than the corresponding upper bound.     Moreover,  as different $\lambda$ makes no difference in most of our experiments, we fine tune $\delta$ instead of precisely calculating the upperbound given by  $\mathrm{L}$ and $\lambda$.  In most of our experiments, $\delta$ is chosen to be smaller than $(1 / 10)\mathrm{L}_y^{-1}$, to ensure convergence of different algorithms. Our choice of $\delta$ is slightly different from PnP-ULA because the Lipschitz parameter $\mathrm{L}$ of the PnP prior can be set to 1 during training.

\noindent\textbf{Implementations}: We implement all the experiments in \texttt{Python} and utilize \texttt{PyTorch} for implementing the ULA Markov chains. The numerical experiments are run on Intel(R) Xeon(R) 
Platinum 8358P CPU with four Nvidia Tesla A100 GPUs. Codes for NF-ULA are available at Github\footnote{\href{https://github.com/caiziruo/NF-ULA}{https://github.com/caiziruo/NF-ULA}}. 

\subsection{Image Deblurring}\label{sec:deblurring}
We first consider a non-blind motion deblurring problem on human face images. The corresponding forward operator $A$ applies a convolution on the image $x$ with a $9\times 9$ motion-blurring kernel of horizontal blurring direction, with all the elements in the fifth row of the kernel being $1/9$ and the other rows being $0$.   Both $x, y\in \mathbb{R}^d$, where $d = 3\times128\times128$ and the forward operator $A: \mathbb{R}^{d} \mapsto \mathbb{R}^{d}$ is linear. To describe the forward model (likelihood), we add Gaussian noise $n\sim \mathcal{N}(0, \sigma^2\,I^d)$, leading to the following measurement equation and likelihood:
\begin{equation*}
    y = Ax + n, \quad p(y| x) = \dfrac{1}{\left(2 \pi \sigma^2\right)^{d/2}}\exp\left( - \dfrac{\left\| y - Ax \right\|^2}{2\sigma^2}  \right).
\end{equation*}
\noindent\textbf{Network training}: To realize NF-ULA, we train the well-known flow-based model, Glow \cite{kingma2018glow}, on the human face dataset FFHQ \cite{karras2019style} without the first $20$ images, which amounts to $69980$ images in total. All the images are 3-channel images normalized to $C_R = [0, 1]^{3\times 128 \times 128}$. We train Glow from scratch using the publicly available \texttt{PyTorch} implementation\footnote{\href{https://github.com/rosinality/glow-pytorch}{https://github.com/rosinality/glow-pytorch}}, however, NF-ULA can also use an appropriate pre-trained model. The architecture of Glow has five blocks with 32 flows in each block. 

For PnP-ULA \cite{laumont2022bayesian}, we use the real spectral normalization DnCNN (realSN-DnCNN), which is a Lipschitz-continuous denoiser proposed in \cite{ryu2019plug}.    In order to see the behavior of the denoiser without the Lipschitz constraint, we train both the standard DnCNN \cite{zhang2017beyond} and realSN-DnCNN \cite{ryu2019plug} on the image patches of a 980-image subset of FFHQ. To train the denoiser, we follow the same procedure reported in \cite{laumont2022bayesian}, i.e., we add Gaussian noise with the variance  $ \varepsilon =  (5/255)^2$ on the training data batches. In fact,  we also tested $ \varepsilon =  (15/255)^2 $ or $(25/255)^2$ but the generated samples get lower  PSNR.  To train the standard DnCNN, we directly use the code in the Image Restoration Toolbox\footnote{\href{https://github.com/cszn/KAIR}{https://github.com/cszn/KAIR}}. We keep the default parameter settings to train a 17-layer DnCNN on image patches of size $40\times40$. For realSN-DnCNN, the original implementation\footnote{\href{https://github.com/uclaopt/Provable_Plug_and_Play/}{https://github.com/uclaopt/Provable\_Plug\_and\_Play/}} in \cite{ryu2019plug} only supports training on grayscale images, therefore we modified the code to make it applicable to color images. We also set up the number of network layers as 17 and preprocess the data to patches of size $40\times 40$, while setting the Lipschitz parameter to 1. Although DnCNN and realSN-DnCNN are trained on such a small dataset, they can still obtain a peak signal-to-noise ratio (PSNR) of more than 40 dB on the validation set. In fact, the original implementation in \cite{ryu2019plug} trains the denoiser on a dataset consisting of only 400 images, and increasing the size of the dataset does not necessarily lead to a higher PSNR on the validation set. 

The Glow network that we used for NF-ULA has 100870544 parameters in total, while DnCNN has 559363 parameters and realSN-DnCNN has 558336 parameters. To train 100 epochs, Glow spent up to 100 hours, while DnCNN and realSN-DnCNN spent less than 5 hours. The heavier network and the longer training time for Glow pay off when it comes to reconstruction performance and image quality.

\begin{table}[!t]
\caption{
The behavior of NF-ULA by different Glow and different choices of $C$. The algorithm does not converge For Glow with affine coupling layers. For Glow with additive coupling layers, the algorithm converges better when Glow is trained for more epochs.  
}
\centering
\begin{tabular}{|l|r|r|r|r|r|r|}
\hline
Deblurring & \multicolumn{4}{|l|}{network: Glow.\quad $  \alpha = 1.5$ } \\
\hline
           & coupling layers & epochs  & C               & PSNR      \\
\hline
face1     & \multicolumn{4}{|l|}{ } \\
\hline
NF-ULA    & affine          & 100      & $ [0, 1]^d      $ & \text{divergent} \\
NF-ULA    & additive        & 5        & $ [-100, 100]^d $ & \text{divergent} \\
NF-ULA    & additive        & 5        & $ [0, 1]^d      $ & 26.58    \\
NF-ULA    & additive        & 20       & $ [-100, 100]^d $ & 29.84    \\
NF-ULA    & additive        & 100      & $ [-100, 100]^d $ & 30.42   \\
\hline
\end{tabular}
\label{tab:epochs-comparing}
\end{table}

\begin{figure}[!htbp]
\caption{
    Deblurring by PnP-ULA and NF-ULA. 
    What each row represents is written on left of the rows. PSNR values corresponding to the sample mean are provided in Table \ref{tab:deblurring}. PnP-ULA with standard DnCNN does not converge on face2 and face4. On all four faces, NF-ULA (Glow) yields a higher PSNR (for the sample mean estimator) than PnP-ULA (realSN-DnCNN). The sample mean images also have a better visual quality for NF-ULA.  
    }
    \centering
    \includegraphics[width=0.95\linewidth]{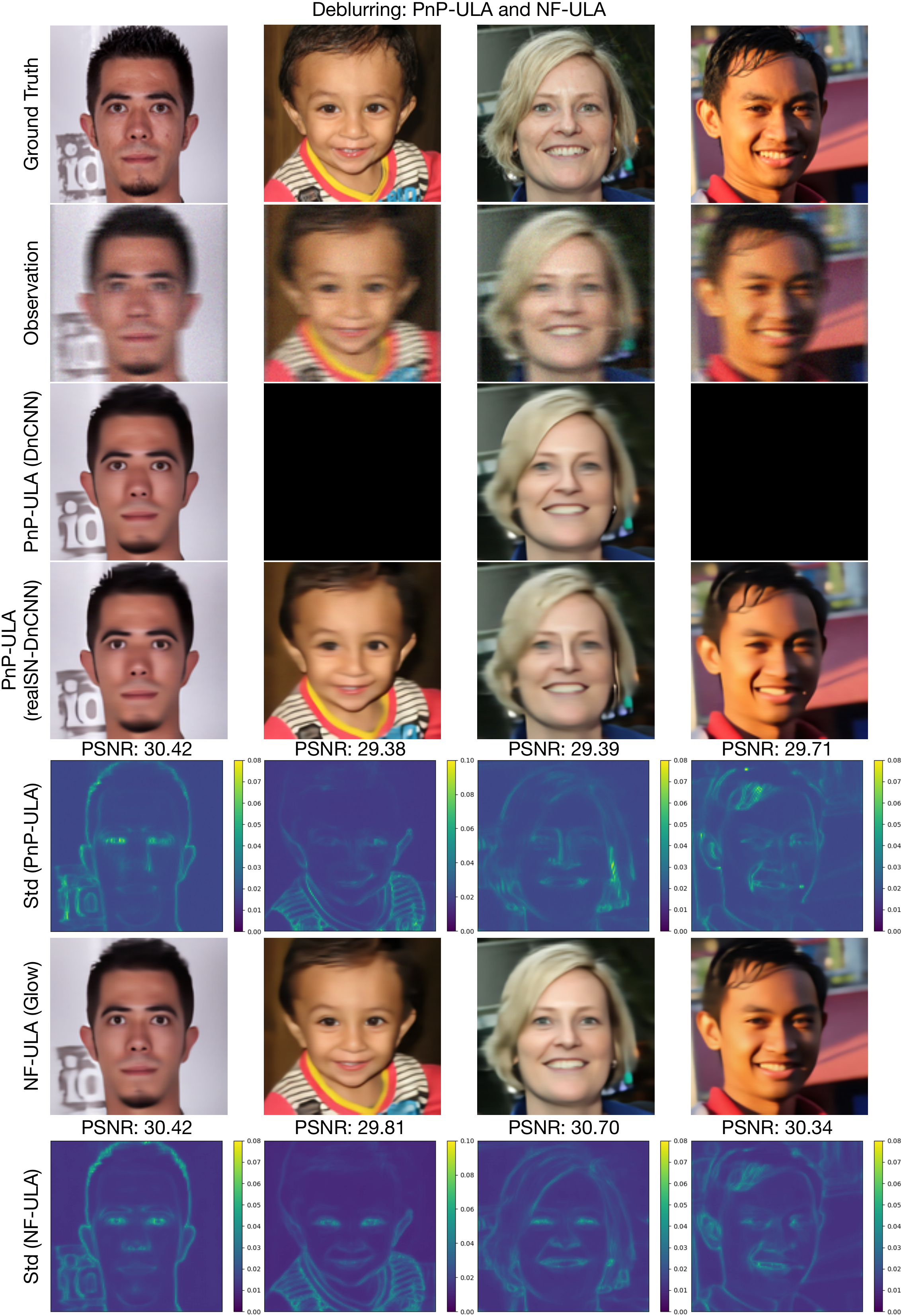}
    \label{fig:deblurring}
\end{figure}

\noindent\textbf{ULA parameters settings}: We set the standard deviation of the Gaussian noise $n$   as $\sigma = 0.02$. To ensure that both PnP-ULA and NF-ULA  are numerically stable, we select the step size $\delta = 5\times 10^{-5}$. For Glow, DnCNN and realSN-DnCNN, $ \alpha = 1.5 $ leads to the highest PSNR. We initialize $X_0 = y$, the noisy blurred observation for both NF-ULA and PnP-ULA.

\noindent\textbf{Performance of the algorithms}: To explore the state space thoroughly, all the experiments have burn-in iterations less than $ 5000$. Since the first sample $X_0$ is initialized as the observation $y$, the PSNR of the samples $X_n$   starts from around $22.78$ dB and then keeps going up and finally stays in an interval, e.g. $[29.0, 31.0]$.   After the burn-in time, we calculate the posterior mean by obtaining $10000$ samples and compute the PSNR of the sample mean. To draw 10000 samples, NF-ULA spends around 3100 seconds,   while PnP-ULA spends 30 seconds.   This is mainly because of the large network Glow uses - the Glow we use has approximately 100 times more parameters than realSN-DnCNN. In fact, we found that computing and forwarding the auto-gradient function of $q_\theta(x)$ takes $10 \%$ longer time than forwarding $q_\theta(x)$ itself. However, we believe that NF-ULA has great potential to leverage smaller and more advanced normalizing flows to reduce computational time.     In Sec. \ref{sec:CT}, we use a lightweight NF-based regularizer and the resulting NF-ULA requires significantly less time.  

To examine the Lipschitz continuity of $\nabla \log q_\theta(x)$ for different kinds of coupling layers, we train two different Glow networks for 100 epochs each, with affine and additive coupling layers, respectively. Also, to verify our hypothesis that better training of the normalizing flow prior will imply better samples from NF-ULA, we trained Glow (additive coupling layers) for 5, 20, and 100 epochs, and compared their performance when used in the NF-ULA framework. 
The PSNR values of the sample mean images corresponding to these variants of NF-ULA with different NF-based priors are reported in Table \ref{tab:epochs-comparing}. With affine coupling layers in Glow, NF-ULA fails to converge because $\nabla \log q_\theta(x)$ is not Lipschitz continuous, which is consistent with Proposition \ref{prop:Lipschitz}. For Glow with additive coupling layers and also for the case where the Glow model is well-trained (more than 20 epochs), NF-ULA works well and the generated samples do not blow up,   even in the case where $C = [-100, 100]^d$ is much bigger than $C_R$. This suggests that a well-trained prior $q_\theta(x)$ already satisfies the tail decay conditions, without imposing the projection $\text{Id} - \Pi_C$. However, it is still essential for the theoretical study. For poorly trained Glow (less than 5 epochs) and large $C$, NF-ULA does not work well - most of the samples go far beyond $C_R$ and the PSNR of them are below $10$ dB. If $C$ is set to be a much smaller set, e.g., $C = C_R$, then the PSNR can be up to 26 dB, which is still considerably lower than what one can achieve with a well-trained Glow.

Intuitively, $q_\theta(x)$ is more \textit{diffusive} when Glow is trained for only a few epochs. After training for some epochs, the normalizing flow is more suitable to serve as an image prior, and the density $q_\theta(x)$ is more concentrated. Moreover, the tail decay condition of $p(x| y)$ is also satisfied with a well-trained prior, even without the projection term. 

\begin{table}[!t]
\caption{Deblurring: Comparison of ULA with different priors for image deblurring. PnP-ULA with a standard DnCNN does not converge on face2 and face4. NF-ULA (Glow) generates samples with slightly higher PSNR than PnP-ULA.}
\centering
\begin{tabular}{|l|r|r|r|r|r|}
\hline
Deblurring  & \multicolumn{3}{|l|}{ net\_epochs = 100, $C = [-100, 100]^d$ }  \\
\hline
        & network      &  parameters       & PSNR  \\
\hline
face1   &     \multicolumn{3}{|l|}{ } \\
\hline
NF-ULA  & Glow         & $  \alpha = 1.5$  & 30.42 \\
\hline 
PnP-ULA & DnCNN        & $  \alpha = 1.5$  & 30.40 \\
PnP-ULA & realSN-DnCNN & $  \alpha = 1.5$  & 30.42 \\
\hline
face2   &     \multicolumn{3}{|l|}{ } \\
\hline
NF-ULA  & Glow         & $  \alpha = 1.5$  & 29.81 \\
\hline 
PnP-ULA & DnCNN        & $  \alpha = 1.5$  & \text{divergent} \\
PnP-ULA & realSN-DnCNN & $  \alpha = 1.5$  & 29.38 \\
\hline
face3   &     \multicolumn{3}{|l|}{ } \\
\hline
NF-ULA  & Glow         & $  \alpha = 1.5$  & 30.70 \\
\hline 
PnP-ULA & DnCNN        & $  \alpha = 1.5$  & 29.61 \\
PnP-ULA & realSN-DnCNN & $  \alpha = 1.5$  & 29.39 \\
\hline
face4   &     \multicolumn{3}{|l|}{ } \\
\hline
NF-ULA  & Glow         & $  \alpha = 1.5$  & 30.34 \\
\hline 
PnP-ULA & DnCNN        & $  \alpha = 1.5$  & \text{divergent} \\
PnP-ULA & realSN-DnCNN & $  \alpha = 1.5$  & 29.71 \\
\hline
\end{tabular}
\label{tab:deblurring}
\end{table}

To compare the performance of ULA with both PnP- and normalizing flow-induced priors, we run NF-ULA using Glow, PnP-ULA using DnCNN, and PnP-ULA with realSN-DnCNN on four human face images randomly selected from the first 20 images of FFHQ \cite{karras2019style}, which are the ones not used during training. In the following experiments, we use Glow with additive coupling layers. Glow, DnCNN, and realSN-DnCNN are all trained for 100 epochs for a fair comparison. The results are shown in Figure \ref{fig:deblurring} and Table \ref{tab:deblurring}. From Table \ref{tab:deblurring}, we note that NF-ULA with Glow generates samples with the highest PSNR. We also present the standard deviation of the samples on the same channel in Fig \ref{fig:deblurring}. NF-ULA has richer details for the posterior mean and more variations for standard deviation, particularly on the eyes, mouths, and hair. This is probably due to a more accurate prior learned by the generative model. It is worth noting that PnP-ULA with DnCNN shows great performance on Face-1 and Face-3, but is divergent on Face-2 and Face-4. However, PnP-ULA with realSN-DnCNN converges on all images, albeit with lower PSNR than NF-ULA.

We record the PSNR of the samples and the minimum mean square error (MMSE) estimator in Figure \ref{fig:PSNR_x_mmse}. It's about the deblurring experiments of face1 and the evolutions for face2, face3, and face4 are similar. In the left figure, we start from the burn-in period until $15000$ samples. The  MMSE estimator is approximated by the last $10000$ samples. For both algorithms, the burn-in periods are less than $5000$ samples. Regardless of the sampling time, NF-ULA shows a faster increase of PSNR, which means the convergence speed of the first-order moment for NF-ULA mildly outperforms PnP-ULA. However, in the right figure, we consider evolution w.r.t. the sampling time and NF-ULA has a slower increase of PSNR. NF-ULA has a burn-time of about 400 seconds while PnP-ULA is less than 40 seconds.

\begin{figure}[!t]
    \centering
    \includegraphics[width=0.8\linewidth]{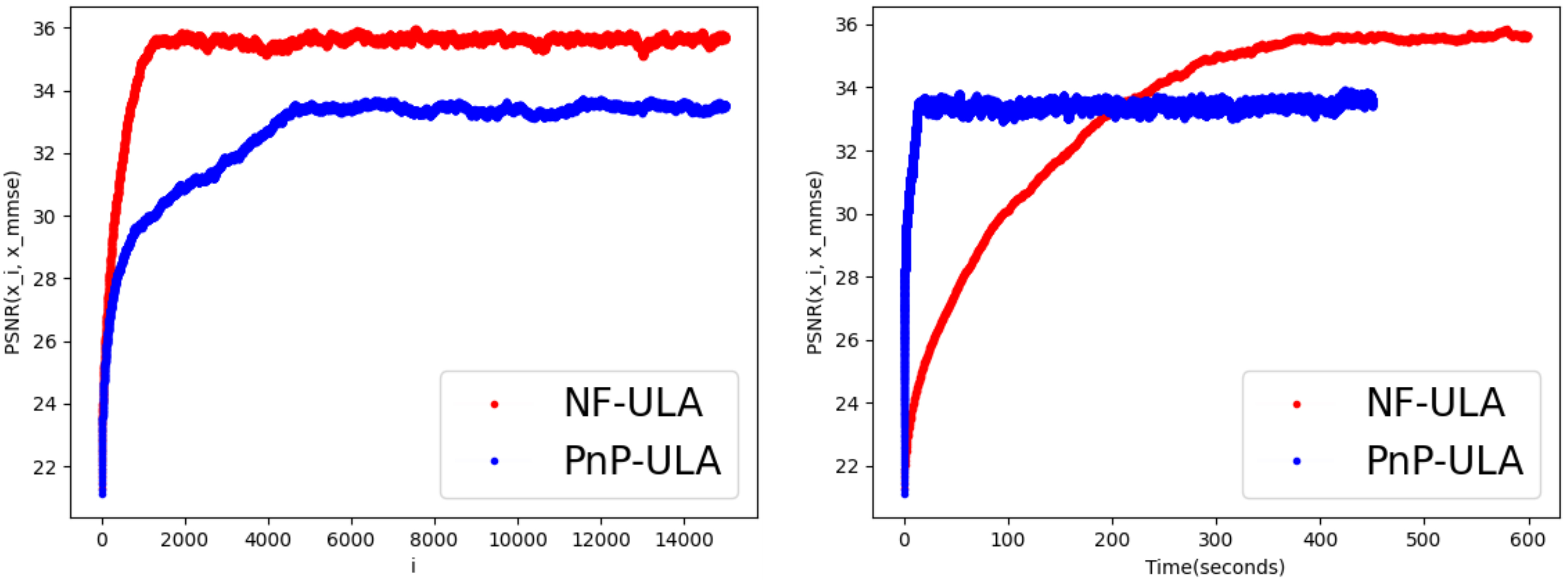}
    \label{fig:PSNR_x_mmse}
    \caption{
    The evolution of PSNR($x_i, x_{\mathrm{mmse}}$) of deblurring (face1). The left figure is according to the number of the samples and the right one is according to elapsed time. A faster increase means a faster convergence speed.
    }
\end{figure}

\begin{figure}[!t]
    \centering
    \includegraphics[width=1.0\linewidth]{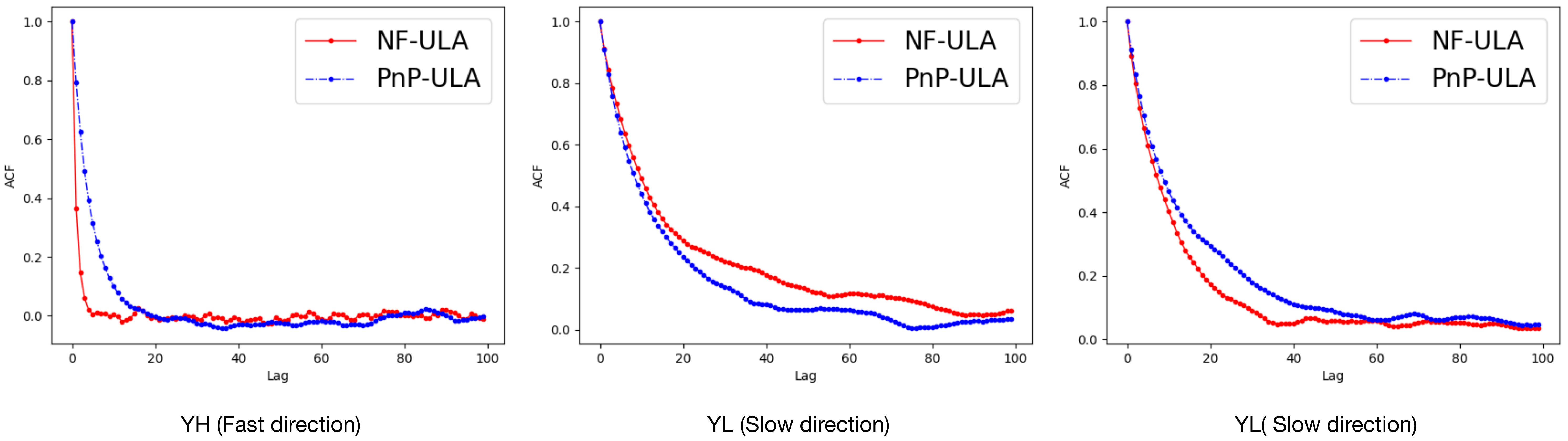}
    \label{fig:acf_deblurring}
    \caption{The autocorrelation function (ACF) of the samples (deblurring on face1). The definition of the ACF is given in \eqref{def:acf}.  ACF is calculated by wavelet basis using the band-pass coefficients (YH) and the low-pass coefficients (YL).  Faster decreasing ACF implies faster convergence of the Markov chain.
    }
\end{figure}

One common approach to studying the convergence speed of a Markov chain is to calculate the $d$-dimensional auto-correlation function (ACF) of it. For samples $\{Y_i\}_{i = 1}^{N}$ from a one-dimensional Markov chain, the sample auto-correlation function is given by
\begin{equation}
\label{def:acf}
\omega(l) = \dfrac{\sum_{t = 1}^{n - l}(Y_{t + l} - \Bar{Y})(Y_t - \Bar{Y})}{\sum_{t = 1}^{n}(Y_t - \Bar{Y})^2},\quad \Bar{Y} = \dfrac{1}{n}\sum_{t = 1}^{n} Y_t,
\end{equation}
where $l = 0, 1, \cdots, n-1$, is the lag between the samples. Since the samples generated by ULA are not strictly uncorrelated, faster decreasing ACF means that the samples are less correlated and generally implies faster convergence of the Markov chain to some extent. Notably, the calculation of ACF is not easy in high-dimensional problems. Therefore, we firstly transform the image samples  using wavelet basis and obtain the band-pass coefficients (YH) and the low-pass coefficients (YL). YH contains the image details while YL captures the overall image structure. We consider the finest scale coefficients in YH.  To  characterize the Markov chain generated by NF-ULA (Glow) and PnP-ULA (realSN-DnCNN), we randomly select 100 dimensions   respectively from YH and YL, and calculate the ACF on those dimensions. It should be noted that the ACF can have different rates of decay in different directions, therefore it is time-consuming to analyze the ACF of all the image dimensions and calculate the fastest and slowest decreasing direction.  However, ACF in YH mostly have  faster decrease and ACF in YL will have slower decrease. In Fig \ref{fig:acf_deblurring}, we show the convergence of ACF (face1), along one \textit{fast direction} in YH and two \textit{slow directions} in YL. In the fast direction, the ACF of PnP-ULA decreases from 1 to 0 within about 20 lags, while for NF-ULA it converges even faster (within approx. 10 lags). For slow directions, both NF-ULA and PnP-ULA hold a non-zero ACF until more than 40 lags, and it is not immediately clear which of these two methods has a faster decay of the ACF.  ACF of face2, face3 and face4 are similar as face1 and hence omitted here.

\subsection{Image Inpainting}\label{sec:inpainting}
\begin{figure}[!htbp]
\caption{
    Comparison of image inpainting performance of PnP-ULA and NF-ULA.
    The PSNR values of the sample mean images are reported in Table \ref{tab:inpainting}. NF-ULA (Glow) yields a higher PSNR (by approximately 2.5-3.0 dB) of the sample mean images than PnP-ULA with a realSN-DnCNN denoiser. This experiment underscores the importance of stronger regularization (which the Glow-based prior can achieve) when the forward operator is severely ill-posed. 
    }
    \centering
    \includegraphics[width=0.95\linewidth]{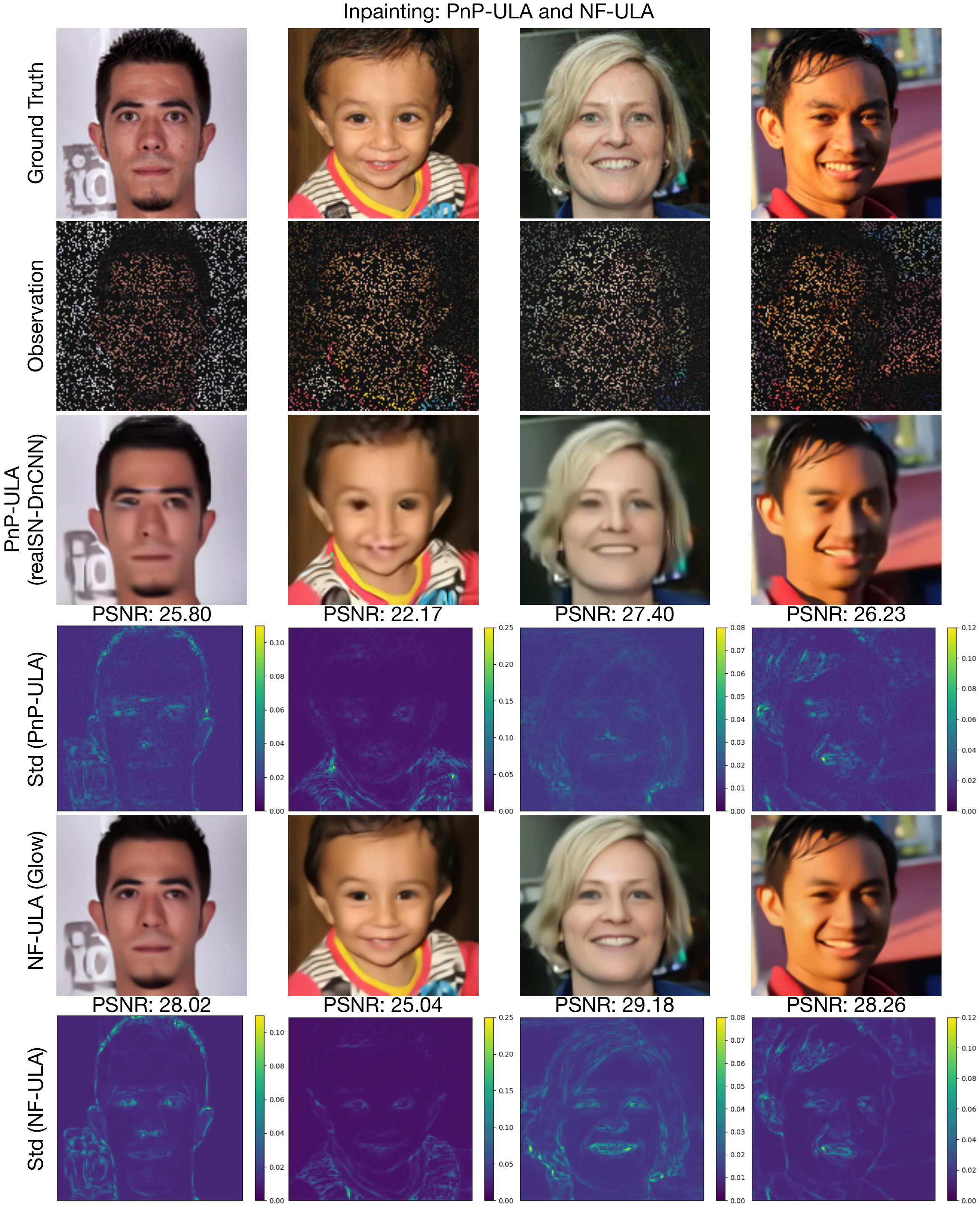}
    \label{fig:inpainting}
\end{figure}

\begin{table}[!t]
\caption{Inpainting: Comparison of the ULA with different priors. The parameter $\alpha$ is fine-tuned to maximize the PSNR for both algorithms. Since inpainting relies more on the prior, NF-ULA has a higher PSNR for the sample mean as compared with PnP-ULA.}
\centering
\begin{tabular}{|l|r|r|r|r|r|}
\hline
Inpainting  & \multicolumn{3}{|l|}{ net\_epochs = 100, $C = [-100, 100]^d$ }  \\
\hline
         & network       &  parameters      & PSNR  \\
\hline
face1   &     \multicolumn{3}{|l|}{ } \\
\hline
NF-ULA  & Glow          & $  \alpha = 2$    & 28.02 \\
\hline 
PnP-ULA & realSN-DnCNN  & $  \alpha = 2.5$  & 25.80 \\
\hline
face2   &     \multicolumn{3}{|l|}{ } \\
\hline
NF-ULA  & Glow          & $  \alpha = 2$    & 25.04 \\
\hline 
PnP-ULA & realSN-DnCNN  & $  \alpha = 2.5$  & 22.17 \\
\hline
face3   &     \multicolumn{3}{|l|}{ } \\
\hline
NF-ULA  & Glow          & $  \alpha = 2$    & 29.18 \\
\hline 
PnP-ULA & realSN-DnCNN  & $  \alpha = 2.5$  & 27.40 \\
\hline
face4   &     \multicolumn{3}{|l|}{ } \\
\hline
NF-ULA  & Glow          & $  \alpha = 2$    & 28.26 \\
\hline 
PnP-ULA & realSN-DnCNN  & $  \alpha = 2.5$  & 26.23 \\
\hline
\end{tabular}
\label{tab:inpainting}
\end{table}

\begin{figure}[!t]
    \centering
    \includegraphics[width=1.0\linewidth]{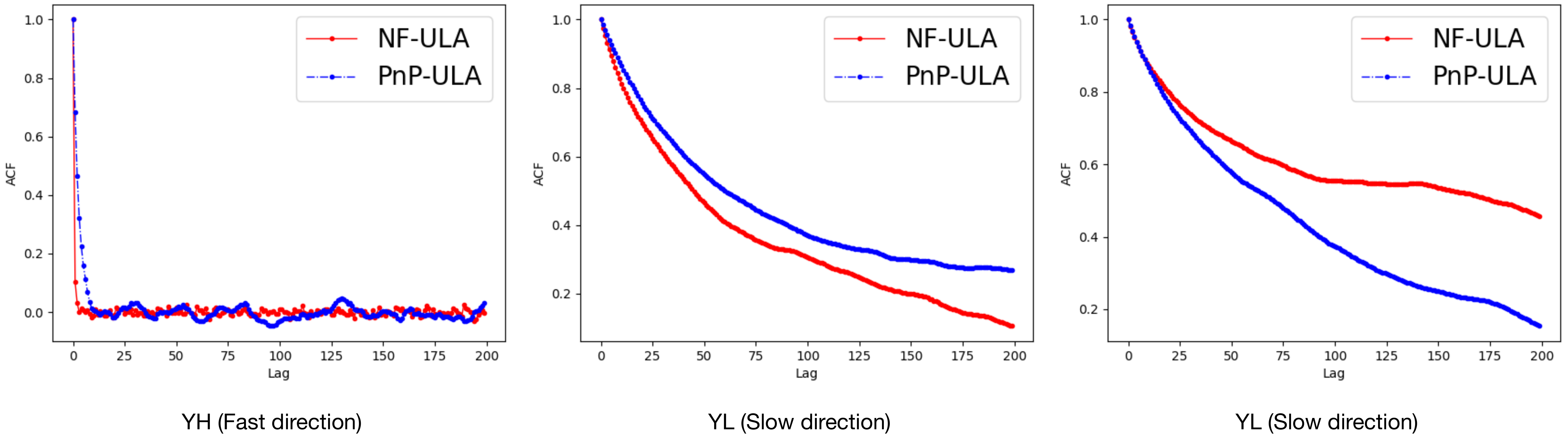}
    \caption{
    The auto-correlation function (ACF) of the samples (inpainting on face1).
    The definition of ACF is given in \eqref{def:acf}. ACF is calculated by wavelet basis using the band-pass coefficients (YH) and the low-pass coefficients (YL). Faster decreasing ACF implies faster convergence of the Markov chain.
    }
    \label{fig:acf_inpainting}
\end{figure}

In this section, we present the experimental results on image inpainting. We still consider human face images and use the Glow and realSN-DnCNN networks trained as explained in Sec. \ref{sec:deblurring}. For inpainting, the forward operator $A$ applies masking on $x$ so that $80\%$ of the pixels in $x$ are missing. We choose different $\alpha $ to ensure both NF-ULA and PnP-ULA have the best performance: $\alpha = 2.0$ works well for NF-ULA, while for PnP-ULA $\alpha = 2.5$ works the best. We maintain the same setting for the other important hyper-parameters of the experiment, such as the noise standard deviation $\sigma = 0.02$, the dimension of image and observation $x,y\in \mathbb{R}^d = \mathbb{R}^{3\times128\times128}$, the step-size of both algorithms $\delta = 5\times 10^{-5}$, the convex set $C = [-100, 100]^d$, and the initialization $X_0 = y$.

\noindent\textbf{Performance of the algorithms}: In contrast with deblurring, we found that both NF-ULA and PnP-ULA have much longer burn-in times.   We initialize $X_0$ with the measurement $y$, whose PSNR is only $5.46$ dB. NF-ULA has a burn-in iteration of $10000$ until the PSNR of $X_n$ grows more than 25 dB and becomes stable, while PnP-ULA takes about $80000$-iterations (eight times larger than NF-ULA) for burn-in. The reason might be that Glow's powerful prior information accelerates the burn-in process, particularly on the pixels missing in the observation. After the burn-in time, we draw $10000$ samples and compute the PSNR of the samples' mean. Drawing 10000 samples takes approximately the same time as in the deblurring experiment.

The sample mean images and the standard deviations are shown in Fig. \ref{fig:inpainting}. As compared with PnP-ULA, NF-ULA recovers more areas of the face and shows higher uncertainties on eyes, hairs, noses, and teeth. Those areas are easily distinguishable between different human faces and should have higher uncertainties than other areas, e.g., foreheads and cheeks. From Table \ref{tab:inpainting}, we observe that NF-ULA achieves a higher PSNR than PnP-ULA. For both NF-ULA and PnP-ULA, the PSNR of the posterior mean is lower than that of the deblurring experiment - the forward operator of masking 80\% pixels is not invertible and the observation $y$ in inpainting is ill-conditioned, which means that in the Bayesian setting, the samples rely on the prior than the likelihood. In such cases, NF-ULA provides a stronger and more informative prior as compared to PnP-ULA.

To calculate the ACF in this inpainting results, we use the same strategy as in deblurring:  calculating the ACF respectively on 100 randomly selected dimensions of YH and YL. In Fig. \ref{fig:acf_inpainting}, we show the ACF including one fast direction in YH and two slow directions in YL. Similar to Fig. \ref{fig:acf_deblurring}, among those fast decreasing directions, the ACF of NF-ULA is slightly faster than PnP-ULA and they both decrease from 1 to 0 within 20 lags. 
For slow directions, both algorithms have  slower decreasing ACF than the deblurring experiments and we cannot conclude for which method, the ACF decreases faster. ACF of face2, face3 and face4 are similar as face1 and omitted.
 
\subsection{CT Reconstruction from limited-angle measurements}\label{sec:CT}
\begin{figure}[!htbp]
    \caption{
    CT reconstruction of Gaussian noise (limited angles). 
    What each column represents is written on top of the columns. PSNR of the samples mean are provided in Table \ref{tab:CT_Gaussian_noise_limited}. NF-ULA (patchNR) yields higher PSNR of samples mean and better samples Std than PnP-ULA (realSN-DnCNN).
    }
    \centering
    \includegraphics[width=0.95\linewidth]{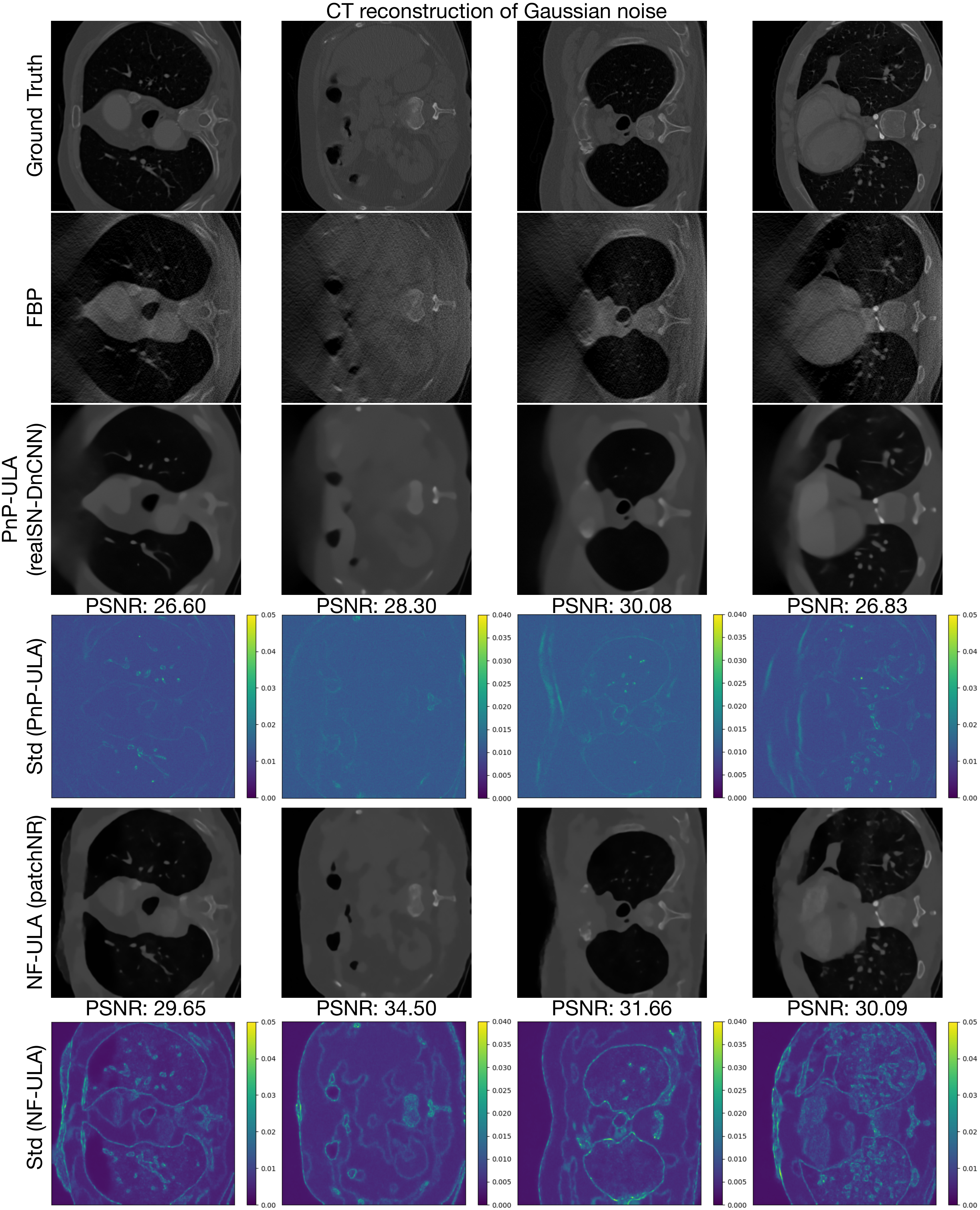}
    \label{fig:CT_Gaussian_noise_limited}  
\end{figure}

\begin{table}[!t]
\caption{Limited-angle CT reconstruction from Gaussian noise-corrupted limited-angle projection data. 
$\alpha$ is chosen to maximize the PSNR for both PnP-ULA and NF-ULA to make a fair comparison. NF-ULA leads to a higher sample mean PSNR than PnP-ULA.
}.  
\centering
\begin{tabular}{|l|r|r|r|r|r|r|r|}
\hline
CT      & \multicolumn{3}{|l|}{ $C = [-100, 100]^d$ } \\
\hline
        & network       &  parameters       & PSNR      \\
\hline
Image-1 &     \multicolumn{3}{|l|}{ } \\
\hline
NF-ULA  & PatchNR       & $  \alpha = 5000$ & 29.65   \\
\hline 
PnP-ULA & realSN-DnCNN  & $  \alpha = 3$    & 26.60     \\
\hline
Image-2 &     \multicolumn{3}{|l|}{ } \\
\hline
NF-ULA  & PatchNR       & $  \alpha = 5000$ & 34.50   \\
\hline 
PnP-ULA & realSN-DnCNN  & $  \alpha = 3$    & 28.30    \\
\hline
Image-3 &     \multicolumn{3}{|l|}{ } \\
\hline
NF-ULA  & PatchNR       & $  \alpha = 5000$ & 31.66   \\
\hline 
PnP-ULA & realSN-DnCNN  & $  \alpha = 3$    & 30.08    \\
\hline
Image-4 &     \multicolumn{3}{|l|}{ } \\
\hline
NF-ULA  & PatchNR       & $  \alpha = 5000$ & 30.09   \\
\hline 
PnP-ULA & realSN-DnCNN  & $  \alpha = 3$    & 26.83    \\
\hline
\end{tabular}
\label{tab:CT_Gaussian_noise_limited}
\end{table}

We consider the classical ill-posed inverse problem of X-ray CT reconstruction from limited-angle projection data. We use the \texttt{torch\_radon} library \cite{torch_radon} to model the forward operator $A$ that computes projections using a fan-beam acquisition geometry. Instead of considering the full angular range $[0, 2\pi]$, we  only have projection data corresponding to an angular sweep over the range $[0.1\,\pi, 0.9\,\pi]$ of angles. We set the number of detector elements to $144$, and test the algorithms for both Gaussian noise and Poisson noise (see Appendix \ref{sec:CT-poisson}). The noisy projection data is given by
\begin{equation}
\label{eq:forward_model_CT}
y = Ax + n \text{ or } y \sim P(Ax),
\end{equation}
where $n$ is used to denote additive Gaussian noise and $P(Ax)$ denotes adding a non-additive noise on $Ax$ such as Poisson noise.  The image to be recovered is $x\in \mathbb{R}^{362\times362}$ and the sinogram is $y\in \mathbb{R}^{144\times 512}$. 
We calculate the norm of $A$ and obtain that $\|A\|=\underset{x:\|x\|=1}{\sup}\,\|Ax\| \approx 100$.

\noindent\textbf{Network architecture}: The features and textures of medical images are more difficult to learn as compared with those in natural images. Hence, normalizing flows do not have comparable performance in generating semantically meaningful images for medical imaging applications, unlike applications involving natural images.  Therefore,  we utilize \textit{patchNR} \cite{altekruger2022patchnr}, which is analogous to normalizing flow,   to apply NF-ULA for CT reconstruction.  PatchNR is a powerful regularizer that involves Glow coupling layers learned on small patches extracted from very few images (only six images), which has shown promising results for CT reconstruction \cite{altekruger2022patchnr}. PatchNR uses  five GlowCoupling blocks and permutations in an alternating
manner, where the coupling blocks are from the FrEIA package \cite{freia}. The three-layer subnetworks are fully connected with ReLU activation functions and 512 nodes, which overall result in a much smaller network than Glow. It should be noted that extracting the patches from an image is not a reversible process, therefore patchNR actually learns the prior over the image patches and cannot do unconditional sampling using $x = T(z)$. Even so, the log gradient is still computable and Lipschitz continuous, since its GlowCoupling blocks satisfy Proposition \ref{prop:Lipschitz}.

The patchNR we used is given by the pre-trained model\footnote{\href{https://github.com/FabianAltekrueger/patchNR}{https://github.com/FabianAltekrueger/patchNR}} trained on six images from the LoDoPaB dataset \cite{leuschner2021lodopab}. For PnP-ULA, we train the denoiser realSN-DnCNN on a 128-image subset of LoDoPaB, by adding Gaussian noise with the variance  $ \varepsilon =  (5/255)^2$ on the training data batches. We train a 17-layer realSN-DnCNN on the preprocessed image patches with size $40 \times 40$.    The Lipschitz parameter of the realSN-DnCNN is set to 1. The patchNR has 2908880 parameters in total and the realSN-DnCNN has 556032 parameters. 

\noindent\textbf{ULA parameters settings}: While in \cite{altekruger2022patchnr} $\alpha = 700$ is the default setting of the considered maximum a posteriori estimator, $\alpha = 5000$ (Gaussian noise) works fine for NF-ULA. For PnP-ULA we set $\alpha = 3$.    We use a smaller step size for both algorithms, namely $\delta = 10^{-6}$, to ensure convergence, since in CT reconstruction the forward operator $A$ has a larger norm (approximately 100) than deblurring and inpainting.   The convex set is set to be $C = [-100, 100]^d$. We initialize $X_0$  using the filtered back-projection (FBP) reconstruction.

\noindent\textbf{Gaussian noise-corrupted measurement}: We first test the case with additive Gaussian noise. To be more specific, we add Gaussian noise
$n\sim \mathcal{N}(0, \sigma^2\,I^m)$ in \eqref{eq:forward_model_CT} to the clean projection data. Since $ \|A \| \approx 100 $,   we select $\sigma = 1.0$ to simulate the noisy sinogram $y$. The likelihood can be expressed as
\begin{equation}
p(y| x) = \dfrac{1}{\left(2 \pi \sigma^2\right)^{m/2}}\exp\left( - \dfrac{\left\| y - Ax \right\|^2}{2\sigma^2}  \right).
\end{equation}
 Since the gradient of the log-likelihood is not globally Lipschitz for Poisson likelihood, the additional experiments with Poisson noise are moved to Appendix \ref{sec:CT-poisson}. Note that NF-ULA with Poisson likelihood still converges although the assumptions needed for the theoretical guarantees do not hold, which warrants further investigations and we leave it for future work.



\noindent\textbf{Performance of the algorithms}: We test PnP-ULA and NF-ULA on another four images from LoDoPaB \cite{leuschner2021lodopab} which were not used for training the patchNR network utilized by NF-ULA and the realSN-DnCNN denoiser used in PnP-ULA. They are different from the six images trained by patchNR and 128 images trained by realSN-DnCNN. The four ground-truth images used for evaluating the performance of NF-ULA and PnP-ULA for limited-angle CT are shown in the first column of Fig. \ref{fig:CT_Gaussian_noise_limited}. 

Both PnP-ULA and NF-ULA have more than $ 20000 $ burn-in iterations. Since we initialize by setting $X_0 $ equal to the FBP reconstruction, the PSNR of $X_n$ starts from around $21.90$ dB, then slowly increases, and finally stabilizes. Note that for different test images, the burn-in time varies. For Image-2 in Table \ref{tab:CT_Gaussian_noise_limited}, PnP-ULA has 30000 burn-in iterations, and the PSNR of the samples never exceeds 29 dB. In contrast, the PSNR of the samples increases until 33 dB for NF-ULA and finally the burn-in time for NF-ULA on Image-2   is around 70000 iterations.  

After the burn-in time, we calculate the posterior mean and the standard deviation around it by obtaining $10000$ samples and computing the PSNR of the samples' mean. For Gaussian noise,  drawing 10000 samples by NF-ULA takes around 500 seconds, whereas, for PnP-ULA, it takes about 70 seconds.   Thanks to the smaller network size of patchNR compared to Glow, it saves a large proportion of time in computation.  

Fig. \ref{fig:CT_Gaussian_noise_limited} shows the ground-truth images (1st column), the FBP (2nd column), the posterior mean and standard deviation of PnP-ULA (in Columns 3 and 4, respectively), and those corresponding to NF-ULA (in Columns 5 and 6, respectively). The posterior mean images indicate that NF-ULA has a significantly better sample quality than PnP-ULA, which exhibits poor reconstruction in the left area, due to the missing angles and the extremely ill-posed problem. NF-ULA can recover the details well, which is consistent with the results in \cite{altekruger2022patchnr} that patchNR works well in the limited-angle CT experiments. For standard deviation in the case of Gaussian noise, NF-ULA shows more realistic uncertainties than PnP-ULA in most areas but still has relatively large uncertainties in the left area (where no projection is available). Table \ref{tab:CT_Gaussian_noise_limited} shows the PSNR of the posterior mean. NF-ULA achieves a considerably higher PSNR than PnP-ULA. 

\begin{figure}[!t]
    \centering
    \includegraphics[width=1.0\linewidth]{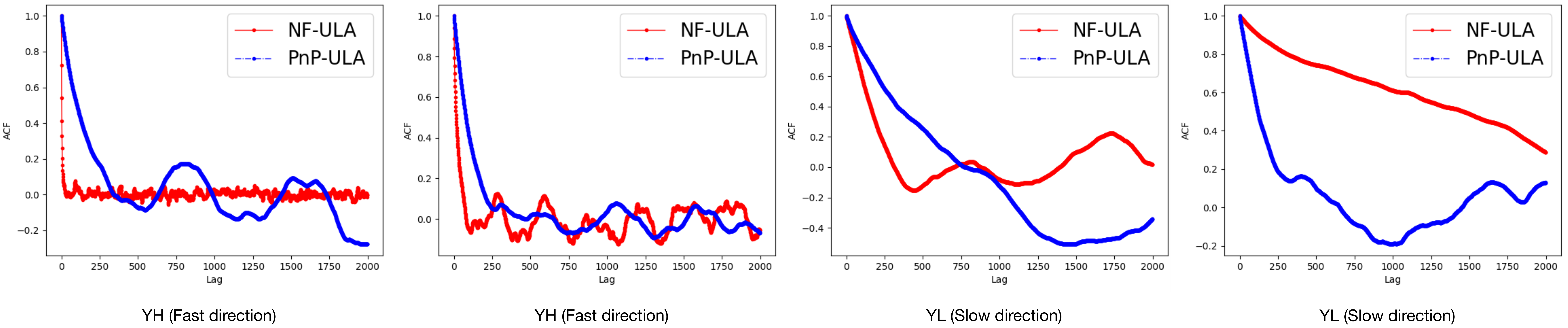}
    \caption{
    The autocorrelation function (ACF) of the samples (Gaussian noise CT on Image-1). The definition of ACF is given in \eqref{def:acf}. ACF is calculated by wavelet basis using the band-pass coefficients (YH) and the low-pass coefficients (YL). Faster decreasing ACF implies faster convergence of the Markov chain. On slow directions, the independence are not achieved for both algorithms.
    }
    \label{fig:acf_CT}
\end{figure}

We also compare the ACF (Image-1) in Fig. \ref{fig:acf_CT} to study the convergence speed. The ACF is calculated by randomly selecting $100$ dimensions respectively from YH and YL. The ACF on the fast direction is different from deblurring and inpainting: On fastest directions NF-ULA decreases from $1$ to $0$ within $100$ lags and the independence is achieved, while the independence of PnP-ULA is not achieved (as shown in the first sub-figure). On some fast directions, the independence of NF-ULA and PnP-ULA is both not well achieved, as demonstrated in the second sub-figure. For slow directions, both two algorithms decrease slowly and independence is not achieved.  ACF of Image-2, Image-3 and Image-4 are similar and omitted.

\section{Conclusion and Outlook}\label{sec:conclusions}
We introduced NF-ULA, a Langevin diffusion-based Monte Carlo algorithm, which takes advantage of a normalizing flow for prior density estimation. The normalizing flow can be pre-trained agnostic to the forward operator of the inverse problem that one seeks to solve. Since NF-ULA only requires the log gradient of the prior, our algorithm still works in cases where the normalizing flow can only evaluate the density but cannot do unconditional sampling.   To guarantee that the posterior distribution is well-defined, we follow \cite{laumont2022bayesian} to add a projection operator onto a convex and compact subset of the image space, although in most cases the projection is not activated, for instance, if the prior is well-trained. Since the density  of normalizing flow itself   can be evaluated, NF-ULA can be extended to a Metropolis-adjusted version, which is left for future studies. For the theoretical analysis of NF-ULA, we first prove the well-posedness of the posterior distribution that we aim to draw samples from. To prove the convergence of NF-ULA, the most essential condition is the Lipschitz drift, and we, therefore, derive a sufficient condition for having a Lipschitz-continuous gradient of the log density of the normalizing flow. Moreover, we show that NF-ULA admits an invariant distribution, and we give a non-asymptotic bound on the bias. We demonstrate our method through several Bayesian imaging experiments, namely image deblurring, image inpainting, and limited-angle CT reconstruction. We show that better training of the normalizing flows leads to better samples and convergence of NF-ULA. Although currently, NF-ULA has a longer sampling time because of the large network of normalizing flows, it has the potential to use a better and smaller network to reduce the computation in the future.

There are still some unanswered questions about NF-ULA. Although we give a sufficient condition for the gradient of the log density of normalizing flow to be Lipschitz, the condition might be relaxed, or it might even be possible to derive a condition that is both necessary and sufficient. Moreover, given different curvature conditions\cite{majka2020nonasymptotic, de2019convergence}   on the drift other than Lipschitz, the studies of ULA on non-convex potentials have shown different convergence results and they can also be applied to NF-ULA. However, this might require re-training the normalizing flows to enforce such conditions and necessitates further research. Meanwhile when the Lipschitz assumption does not hold, the results of our Poisson noise experiments lack an explanation, which also requires a more detailed study.

\section*{Acknowledgments}
CBS acknowledges support from the Philip Leverhulme Prize, the Royal Society Wolfson Fellowship, the EPSRC advanced career fellowship EP/V029428/1, EPSRC grants EP/S026045/1 and EP/T003553/1, EP/N014588/1, EP/T017961/1, the Wellcome Innovator Awards 215733/Z/19/Z and 221633/Z/20/Z, the European Union Horizon 2020 research and innovation program under the Marie Skodowska-Curie grant agreement No. 777826 NoMADS, the Cantab Capital Institute for the Mathematics of Information and the Alan Turing Institute.  ZC and XZ  were partially supported by NSFC (No. 12090024) and Sino-German center grant (No.M-0187) .

\bibliographystyle{siamplain}
\bibliography{reference}

\appendix
\section{Proofs}

\subsection{Proof of Lemma \ref{lemma:finite-moment}}
\label{proof:finite-moment}
\begin{proof}
For a constant $R_0>0$, let 
\begin{equation*}
    B(0,R_0) := \left\{z\in\mathbb{R}^d:\left\|z\right\|_2\leq R_0\right\}
\end{equation*}
be the closed ball of radius $R_0$ centered at the origin. Since $C\subset\mathbb{R}^d$ is compact, there exists $R_0>0$ such that $C\subset B(0,R_0)$. Therefore, for all $x\notin B(0,R_0)$, it follows that
\begin{align*}
\left\| x - \Pi_C(x) \right\|_2 &\stackrel{\text{(a)}}{\geqslant} \left\| x - \Pi_{B(0,R_0)}(x) \right\|_2\stackrel{\text{(b)}}{\geqslant}  \|x\|_2-R_0\geq 0,
\end{align*}
where (a) is true since $C\subset B(0,R_0)$ and (b) follows from the triangle inequality.
Then, for all $k \in \mathbb{N}$, the following holds:
\begin{equation*}
\begin{aligned}
& \int_{\mathbb{R}^d \setminus B(0, R_0)} \left\| x \right\|^k  \exp\left( - \dfrac{\left\| x - \Pi_C(x) \right\|_2^2}{2\lambda}\right) \mathrm{d} x 
\\
& \leqslant \int_{\mathbb{R}^d  \setminus B(0, R_0)} \left\| x \right\|^k  \exp\left( - \dfrac{  \left( \left\| x \right\|_2 - R_0 \right)^2  }{2\lambda}\right) \mathrm{d} x \\
& \leqslant \int_{\mathbb{R}^d  \setminus B(0, R_0)} \left\| x \right\|^k  \exp\left( - \dfrac{ \|x\|^2_2 - 2R_0^2    }{4\lambda}\right) \mathrm{d} x \,\,\,  \\
& < +\infty,
\end{aligned}
\end{equation*}
where the last inequality follows from the fact that $k$-order moments of Gaussian distribution are finite for any $k$.
\end{proof}


\subsection{Proof of Proposition \ref{prop:bounded_q}}\label{proof:bounded_q}
\begin{proof}
Without loss of generality, we only need to consider the cases when the total number of layers is $ k= 1, 2$. 

(1) We firstly consider the case that $k = 1$  and $T^{-1} = G$ is a composition of only a one-layer  coupling network. Then \eqref{eq:struct_of_inv_map1} can be simplified as:
\begin{equation}
    G_j(x_j, x_{<j}) = \varphi_j(x_{<j}) x_j + \eta_j(x_{<j}), \,\,j = 1,\cdots, d.
\end{equation}
Since  $\forall r<j$, $G_{r} $ is independent of $x_j$  and the diagonal of the Jacobian is $(J_G(x))_{j, j} = \varphi_j(x_{<j})$, from the change of variables
\begin{equation}
\label{NF:change of variable1}
\begin{aligned}
    q(x)&=q_{z}(z)\left|\operatorname{det} J_{T}(z)\right|^{-1} \quad   \\
    &=q_{z}\left(T^{-1}(x)\right)\left|\operatorname{det} J_{T^{-1}}(x)\right|,
\end{aligned}
\end{equation}
we have that
\begin{equation*}
\begin{aligned}
    \log q_\theta(x) 
    & =  \log q_{z}\left(G(x)\right) + \log \left|\operatorname{det} J_{G}(x)\right| 
    \\
    & = - \dfrac{1}{2}\left\| G(x) \right\|_2^2 + \log \left|\operatorname{det} J_{G}(x)\right| + \text{const}.
    \\
    &
    = -\dfrac{1}{2}\left\| G(x) \right\|^2_2 + \sum_{j = 1}^{d} \log \left| \varphi_j(x_{<j}) \right| + \text{const}.
    \\
    & \leqslant \sum_{j = 1}^{d} \log \left| \varphi_j(x_{<j}) \right| + \text{const}.
\end{aligned}
\end{equation*}
Since $\varphi_j$ is a bounded function $\forall j$,  it follows that $ \log \left| \varphi_j(x_{<j}) \right|$ is upper bounded for all $j$ and $ \log q_\theta(x)$ is upper bounded on $\mathbb{R}^d$.

(2) Secondly, assume that $k = 2$ and $T^{-1} = G\circ H(x)$, where $H:x\mapsto \omega$ and $G:\omega \mapsto z$. Similarly, we have that
\begin{equation*}
\begin{aligned}
    \log q_\theta(x) 
    & =  \log q_{z}\left(G\circ H(x)\right) + \log \left|\operatorname{det} J_{G\circ H}(x)\right| 
    \\
    & = - \dfrac{1}{2}\left\| G\circ H(x) \right\|_2^2 + \log \left|\operatorname{det} J_{G }(\omega)\right| + \log \left|\operatorname{det} J_{H }(x)\right| + \text{const}.
    \\
    &
    = -\dfrac{1}{2}\left\| G\circ H(x) \right\|^2_2 + \sum_{j = 1}^{d} \left( \log \left| \varphi^{(2)}_j(\omega_{<j}) \right|  + \log \left| \varphi^{(1)}_j(x_{<j}) \right| \right) + \text{const}.
    \\
    & \leqslant \sum_{j = 1}^{d} \left( \log \left| \varphi^{(2)}_j(\omega_{<j}) \right|  + \log \left| \varphi^{(1)}_j(x_{<j}) \right| \right) + \text{const}.
\end{aligned}
\end{equation*}
Since $\varphi^{(1)}_j$ and $\varphi^{(2)}_j$ are bounded functions $\forall j$,  it follows that $ \log \left| \varphi^{(2)}_j(\omega_{<j}) \right| + \log \left| \varphi^{(1)}_j(x_{<j}) \right|$ is upper bounded for all $j$ and $ \log q_\theta(x)$ is upper bounded on $\mathbb{R}^d$.
\end{proof}

\subsection{Proof of Proposition \ref{prop:well-posedness}}\label{proof:well-posedness}
\begin{proof}
By Assumption \ref{Assumption1}, we have that
\begin{equation*}
\int_{\mathbb{R}^d}\left(1+\Phi_1(\tilde{x})\right) \exp \left[c_0 \Phi_1(\tilde{x})    - \iota^{(\lambda)}_C(\tilde{x})  \right] q_\theta^\alpha(\tilde{x})   \mathrm{d} \tilde{x} <+\infty,
\end{equation*}
and we conclude the proof from Proposition 2.3 of \cite{laumont2022bayesian}.

\end{proof}

\subsection{Proof of Lemma \ref{lemma:Lipschitz}}\label{proof:Lipschitz}
\begin{lemma}
Let Assumption \ref{Assumption2} be true. Then, $ \nabla \log p_{\lambda}(x | y)$ is Lipschitz continuous if and only if $ \nabla \log q_\theta(x) $ is Lipschitz continuous.  
\end{lemma}
\begin{proof}
Since Assumption \ref{Assumption2} is satisfied, from Algorithm \ref{NF-ULA} and \eqref{NF-ULA-simplified} we have that $\nabla \log p_\lambda(x| y)$ is Lipschitz continuous if and only if $\alpha \nabla \log q_\theta(x) + (\Pi_C(x) - x) /\lambda $ is Lipschitz continuous.

From Proposition 12.28 in \cite{bauschke2011convex}, the operator $(\text{Id} - \operatorname{Prox}_{\iota_C})$ is firmly non-expansive, i.e., for all $x, y\in \mathbb{R}^d$,
\begin{equation*}
\begin{aligned}
\left\|(\Pi_C(x) - x) - (\Pi_C(y) - y) \right\|_2^2 
& \leqslant \left\langle (\Pi_C(x) - x) - (\Pi_C(y) - y), x - y \right\rangle \\
& \leqslant \left\| (\Pi_C(x) - x) - (\Pi_C(y) - y) \right\|_2 \left\| x - y \right\|_2.
\end{aligned}
\end{equation*}
Therefore,  $(\Pi_C(x) - x) / \lambda$ is $1 / \lambda$-Lipschitz. Hence, for any $\alpha > 0$, $\nabla \log p_\lambda(x| y)$ is Lipschitz-continuous if and only if $\nabla \log q_\theta(x)$ is Lipschitz-continuous.  
\end{proof}

\subsection{Proof of Proposition \ref{prop:Lipschitz}}\label{proof:prop_Lipschitz}
\begin{proof}
Without loss of generality, we only need to consider the cases when the total number of layers is $ k= 1, 2$.

(1) We firstly consider the case that $k = 1$  and $T^{-1} = G$ is a composition of only a one-layer  coupling network. Then \eqref{eq:struct_of_inv_map} can be simplified as:
\begin{equation}
    G_j(x_j, x_{<j}) = \varphi_j(x_{<j}) x_j + \eta_j(x_{<j}), \,\,j = 1,\cdots, d.
\end{equation}
Since  $\forall r<j$, $G_{r} $ is independent of $x_j$  and the diagonal of the Jacobian is $(J_G(x))_{j, j} = \varphi_j(x_{<j})$, from the change of variables
\begin{equation}
\begin{aligned}
    q(x)&=q_{z}(z)\left|\operatorname{det} J_{T}(z)\right|^{-1} \quad   \\
    &=q_{z}\left(T^{-1}(x)\right)\left|\operatorname{det} J_{T^{-1}}(x)\right|,
\end{aligned}
\end{equation}
we have that
\begin{equation*}
\begin{aligned}
    \log q_\theta(x) 
    & =  \log q_{z}\left(G(x)\right) + \log \left|\operatorname{det} J_{G}(x)\right| 
    \\
    & = - \dfrac{1}{2}\left\| G(x) \right\|_2^2 + \log \left|\operatorname{det} J_{G}(x)\right| + \text{const}.
    \\
    &
    = -\dfrac{1}{2}\left\| G(x) \right\|^2_2 + \sum_{j = 1}^{d} \log \left| \varphi_j(x_{<j}) \right| + \text{const}.
\end{aligned}
\end{equation*}
Taking the gradient of both sides w.r.t. $x$, we get
\begin{equation}
\begin{aligned}
& \nabla \log q_\theta(x) = - \left(J_{G}(x) \right)^T   G(x) +  \sum_{j = 1}^{d} \nabla \log  \varphi_j(x_{<j}).  
\end{aligned}
\end{equation}
Since $\varphi_j$ is a constant function, we have that $\nabla \log  \varphi_j = 0$. Furthermore as  $\eta_j$ is Lipschitz and $\displaystyle \forall r < j, ~\frac{\partial \eta_j}{\partial x_r}$ is piecewise constant on $\mathbb{R}$,  $\dfrac{\partial \eta_j}{\partial x_r}$ is hence bounded. Meanwhile,     $\displaystyle (J_G(x))_{j, r} = \frac{\partial \eta_j}{\partial x_r}  $, therefore every element of $J_G(x)$ is a bounded piecewise constant function of $x$. Then both $G(x)$ and $ (J_G(x))^\top G(x)$ are Lipschitz, therefore $ \nabla \log q_\theta(x) $ is  Lipschitz.

(2) Secondly, assume that $k = 2$ and $T^{-1} = G\circ H(x)$, where $H:x\mapsto \omega$ and $G:\omega \mapsto z$. Similarly, we have that
\begin{equation*}
\begin{aligned}
    \log q_\theta(x) 
    & =  \log q_{z}\left(G\circ H(x)\right) + \log \left|\operatorname{det} J_{G\circ H}(x)\right| 
    \\
    & = - \dfrac{1}{2}\left\| G\circ H(x) \right\|_2^2 + \log \left|\operatorname{det} J_{G }(\omega)\right| + \log \left|\operatorname{det} J_{H }(x)\right| + \text{const}.
    \\
    &
    = -\dfrac{1}{2}\left\| G\circ H(x) \right\|^2_2 + \sum_{j = 1}^{d} \left( \log \left| \varphi^{(2)}_j(\omega_{<j}) \right|  + \log \left| \varphi^{(1)}_j(x_{<j}) \right| \right) + \text{const}.
\end{aligned}
\end{equation*}
and 
\begin{equation}
\begin{aligned}
\nabla \log q_\theta(x) 
& = - \left(J_{G\circ H}(x) \right)^T   G\circ H (x) + 0   \\
& = - \left(J_{G}(H(x)) J_{H}(x) \right)^T   G\circ H (x).
\end{aligned}
\end{equation}
Since  every element of $J_H(x)$ is a bounded piecewise constant function of $x$,  every element of $J_G(w)$ is a bounded piecewise constant function of $w$, and meanwhile $w = H(x)$ is  continuous w.r.t. $x$,  then every element of $J_{G\circ H}(x)$ is a bounded piecewise constant function of $x$. Then both $G\circ H(x)$ and $ (J_{G\circ H}(x))^\top G\circ H(x)$ are Lipschitz, therefore $ \nabla \log q_\theta(x) $ is  Lipschitz.
\end{proof}
 
\subsection{Proof of theorem \ref{theorem:contractive}}\label{proof:contractive}


\begin{proof}

Denote $R_{\mathrm{C}}=\sup \left\{\left\|x_1-x_2\right\|: x_1, x_2 \in \mathrm{C}\right\}$. Since we have $2 \lambda(\alpha \mathrm{L} -\mathrm{m}_y) \leqslant 1$, from A \ref{Assumption3}, A \ref{Assumption4}, $b_\lambda(x)$ in (\ref{definition:stochastic-kernel}) and the Cauchy-Schwarz inequality we have that for any $x_1, x_2 \in \mathbb{R}^d$,
\begin{equation}
\label{equation:first-contraction}
\begin{aligned}
\left\langle b_{\lambda}\left(x_1\right)-b_{\lambda}\left(x_2\right), x_1-x_2\right\rangle & \leqslant(-\mathrm{m}_y+\alpha \mathrm{L} )\left\|x_1-x_2\right\|^2- \dfrac{\left\|x_1-x_2\right\|^2}{\lambda} + \dfrac{R_{\mathrm{C}}\left\|x_1-x_2\right\|}{\lambda}  \\
& \leqslant-\dfrac{\left\|x_1-x_2\right\|^2}{2 \lambda}+ \dfrac{R_{\mathrm{C}}\left\|x_1-x_2\right\|}{\lambda}  .
\end{aligned}
\end{equation}
For any $x_1, x_2 \in \mathbb{R}^d$ satisfying $\left\|x_1-x_2\right\| \geqslant 4 R_{\text {C }}$, we obtain the contractivity at infinity condition on the drift $b_\lambda$
\begin{equation}
\label{b-contractivity-condition}
\left\langle b_{\lambda}\left(x_1\right)-b_{\lambda}\left(x_2\right), x_1-x_2\right\rangle \leqslant-  \dfrac{\left\|x_1-x_2\right\|^2}{4 \lambda} ,
\end{equation}
which indicates the strongly convexity at infinity.

After simple computation by letting $x_2 = 0$ in (\ref{equation:first-contraction}), we also have that for any $x \in \mathbb{R}^d$,
\begin{equation}
\label{condition-lemma5.1}
\begin{aligned}
& \left\langle b_{\lambda}(x), x\right\rangle \leqslant-\|x\|^2 /(4 \lambda) +\sup _{\tilde{x} \in \mathbb{R}^d}\left\{\left(R_{\mathrm{C}} / \lambda+\|b_\lambda(0)\|\right)\|\tilde{x}\|-\|\tilde{x}\|^2 /(4 \lambda)\right\} .
\end{aligned}
\end{equation}

From A \ref{Assumption2}, A \ref{Assumption3}, $b_\lambda(x)$ in (\ref{definition:stochastic-kernel}) and that  $ (\operatorname{Id}-\Pi_C)  / \lambda$ is $1/\lambda$-Lipschitz, we have that for any $x_1, x_2 \in \mathbb{R}^d$,
\begin{equation}
\label{eq:Lipschitz_drift_lambda}
\left\|b_{\lambda}\left(x_1\right)-b_{\lambda}\left(x_2\right)\right\|_2 \leqslant\left(\mathrm{L}_y+\alpha \mathrm{L} +1 / \lambda\right)\left\|x_1-x_2\right\|_2 .
\end{equation}

Let $\bar{\gamma} = (4 \lambda)^{-1}\left(\mathrm{~L}_y+\alpha \mathrm{L} +1 / \lambda\right)^{-2} $. From \eqref{condition-lemma5.1}  and \eqref{eq:Lipschitz_drift_lambda},  using Lemma SM5.1  in \cite{SupplementaryMaterials} and  we get that there exist $\lambda_V\in \left(0, 1\right]$, $c\geqslant 0$ such that for any $\delta \in (0, \bar{\gamma}]$, $\mathrm{R}_\delta$ satisfies the discrete drift condition $\mathbf{D}_\mathrm{d}\left(  V, \lambda_V^\delta, c\delta  \right)$.

For any probability measure $\nu_1, \nu_2$, from the definition (\ref{definition:V-norm}) and Hölder's inequality we have that 
\begin{equation}
\label{eq:holder_inequality}
\left\|\nu_1-\nu_2\right\|_V \leqslant\left\|\nu_1-\nu_2\right\|_{\mathrm{TV}}^{1 / 2}\left(\nu_1\left[V^2\right]+\nu_2\left[V^2\right]\right)^{1 / 2} .   
\end{equation}

Since $ \bar{\delta}\leqslant \bar{\gamma}$, the contractivity condition (\ref{b-contractivity-condition}) holds, \eqref{eq:holder_inequality} holds, then from Theorem 8 and Corollary 2 in \cite{de2019convergence}, we can find $A_{2 } \geqslant 0$ and $\rho_{2 } \in[0,1)$ such that for any $\delta \in(0, \bar{\delta}], x_1, x_2 \in \mathbb{R}^d$, and $k \in \mathbb{N}$,
\begin{equation}
\begin{aligned}
\left\|\boldsymbol{\delta}_{x_1} \mathrm{R}_{\delta}^k-\boldsymbol{\delta}_{x_2} \mathrm{R}_{ \delta}^k\right\|_\mathrm{TV} 
& \leqslant A_{2 } \rho_{2 }^{k \delta}\left(V\left(x_1\right)+V\left(x_2\right)\right)     
\\
& \leqslant A_{2 } \rho_{2 }^{k \delta}\left(V^2 \left(x_1\right)+V^2 \left(x_2\right)\right)
\\
\mathbf{W}_1\left(\boldsymbol{\delta}_{x_1} \mathrm{R}_{ \delta}^k, \boldsymbol{\delta}_{x_2} \mathrm{R}_{ \delta}^k\right) 
& \leqslant A_{2 } \rho_{2 }^{k \delta}\left\|x_1-x_2\right\|_2.
\end{aligned}
\end{equation}

Then we conclude the proof from \eqref{eq:holder_inequality}. 
\end{proof}

\subsection{Proof of theorem \ref{theorem:nonasymptotic-bias}}\label{proof:nonasymptotic-bias}

\begin{proof}
Most of our proof is based on  \cite{SupplementaryMaterials} and \cite{de2019convergence}.

Recall that 
\begin{equation}
\begin{aligned}
\mathrm{R}_{ \delta }(x, \mathrm{~A}) 
& = (2 \pi)^{-d / 2} \int_{\mathbb{R}^d} \mathbf{1}_{\mathrm{A}}\left(x+\delta b_\lambda(x)+\sqrt{2 \delta} z\right) \exp \left[-\|z\|^2 / 2\right] \mathrm{d} z.
\end{aligned}
\end{equation}
We introduce the stochastic process $\left(\overline{\mathbf{X}}_t\right)_{t \geqslant 0}$, which is exactly the solution of the following SDE:
\begin{equation}
\label{SDE:construction1}
\left\{
\begin{aligned}
\mathrm{d} \overline{\mathbf{X}}_t 
& =b_{\lambda}\left(\overline{\mathbf{X}}_t\right) \mathrm{d} t+\sqrt{2} \mathrm{~d} \mathbf{B}_t 
\\
b_{\lambda}(x) 
& =\nabla \log (p(y | x))+\alpha \nabla \log q_{\theta}(x) + \dfrac{ \Pi_C\left(x\right) - x }{\lambda}
\\
\overline{\mathbf{X}}_0 
& = X_0
\end{aligned}
\right.
\end{equation}
where $\left(\mathbf{B}_t\right)_{t \geqslant 0}$ is a $d$-dimensional Brownian motion. 

From Lemma \ref{lemma:Lipschitz},  $b_{\lambda}$ is $ (\mathrm{L}_y+\alpha \mathrm{L} +1 / \lambda) $-Lipschitz continuous. From Chapter 5, Theorem 2.9 of \cite{karatzas1991brownian} we have that the  SDE (\ref{SDE:construction1}) admits a unique strong solution for any initial condition $\overline{\mathbf{X}}_0$ with $\mathbb{E}\left[\left\|\overline{\mathbf{X}}_0\right\|^2\right]<$ $+\infty$. We denote by $\left(\mathrm{P}_{t}\right)_{t \geqslant 0}$ the semigroup associated with the strong solutions of SDE (\ref{SDE:construction1}). Similarly to the proof of Theorem \ref{theorem:contractive}, replacing Corollary 2 in \cite{de2019convergence} by Theorem 21 and Corollary 22 in \cite{de2019convergence}, there exist $\tilde{A}_{1} \geqslant 0$ and $\tilde{\rho}_{1} \in[0,1)$  such that that for any $x_1, x_2 \in \mathbb{R}^d$ and $t \geqslant 0$,
\begin{equation}
\label{semigroup-contractive}
\begin{aligned}
& \left\|\boldsymbol{\delta}_{x_1} \mathrm{P}_{t }-\boldsymbol{\delta}_{x_2} \mathrm{P}_{t }\right\|_V \leqslant \tilde{A}_{1} \tilde{\rho}_{1}^t\left(V^2\left(x_1\right)+V^2\left(x_2\right)\right), \\
& \mathbf{W}_1\left(\boldsymbol{\delta}_{x_1} \mathrm{P}_{t }, \boldsymbol{\delta}_{x_2} \mathrm{P}_{t }\right) \leqslant \tilde{A}_{1} \tilde{\rho}_{1}^t\left\|x_1-x_2\right\|_2.
\end{aligned}
\end{equation}
Combining (\ref{semigroup-contractive}), Theorem \ref{theorem:contractive},  the fact that $\left(\mathscr{P}_1\left(\mathbb{R}^d\right), \mathbf{W}_1\right)$ is a complete metric space and the Picard fixed point theorem, we can obtain that for any $\delta \in(0, \bar{\delta}]$ there exist $\pi_{\delta, \lambda}, \tilde{\pi}_{\lambda} \in$ $\mathscr{P}_1\left(\mathbb{R}^d\right)$ such that $\pi_{ \delta, \lambda} \mathrm{R}_{  \delta }=\pi_{ \delta, \lambda}$ and for any $t \geqslant 0, \tilde{\pi}_{\lambda} \mathrm{P}_{t }=\tilde{\pi}_{\lambda}$. By Theorem 2.1 in \cite{roberts1996exponential} we have that for any $x \in \mathbb{R}^d$,
\begin{equation}
\left(\mathrm{d} \tilde{\pi}_{\lambda} / \mathrm{dLeb}\right)(x) \propto \exp \left[- \iota^{(\lambda)}_C(x)  \right] p(y | x) p_{\lambda}^\alpha(x),    
\end{equation}
Therefore from (\ref{definition:posterior_measure}) $\pi_\lambda$ and $ \tilde{\pi}_\lambda $ are exactly the same.




Similar to (\ref{generalized-contractive}), from (\ref{semigroup-contractive}) we have that for any $t\geqslant 0$ and $x\in \mathbb{R}^d$,
\begin{equation}
\left\|\boldsymbol{\delta}_{x} \mathrm{P}_{t }-\pi_\lambda  \right\|_V \leqslant \tilde{A}_{1} \tilde{\rho}_{1}^t\left(V^2\left(x\right)+  \int_{\mathbb{R}^d} V^2(\tilde{x}) \mathrm{d} \pi_\lambda(\tilde{x}) \right)
\end{equation}
Since we already proved that $\int_{\mathbb{R}^d} V^2(\tilde{x}) \mathrm{d} \pi_\lambda(\tilde{x}) < +\infty$ in Lemma \ref{lemma:finite-moment},  we can find $B_1 \geqslant 0$ such that for any $x \in \mathbb{R}^d$  we have
\begin{equation}
\label{error-P}
\left\|\boldsymbol{\delta}_x \mathrm{P}_{t} -  \pi_{\lambda} \right\|_V \leqslant B_1 \tilde{\rho}_{1}^t V^2(x).
\end{equation}

Select a large $m_1 \in \mathbb{N}^*$ such that  $m_1 \geqslant \bar{\delta}^{-1}$.  Let's now consider the interval $[0, l]$, $l \in \mathbb{N}^*$. To compare  $\pi_{\delta, \lambda}$ with $\pi_{\delta }$, we first construct a continuous time Markov process $X_t^{(1)}$ such that $X^{(1)}_{j/m_1}$ has the same distribution as the $j$-th sample $X_j$ by NF-ULA (\ref{NF-ULA-simplified}). 

Define 
$\displaystyle b_1\left(t,\left(w_t\right)_{t \in[0, l]}\right)=\sum_{j=0}^{m_1 l-1} \mathbf{1}_{[j / m_1,(j+1) / m_1)}(t) b_{\lambda}\left(w_{j / m_1}\right)$
  and $ b_2\left(t,\left(w_t\right)_{t \in[0, l]} \right)=b_{\lambda}\left(w_t\right)$. Let $\mathbf{X}_t^{(1)}$ and $\mathbf{X}_t^{(2)}$ be the unique strong solution of  SDE $\mathrm{d} \mathbf{X}_t=b\left(t,\left(\mathbf{X}_t\right)_{t \in[0,l]}\right)\mathrm{d}t +\sqrt{2} \mathrm{d} \mathbf{B}_t$ with $\mathbf{X}_0=x  \in \mathbb{R}^d$ and $b=b_1$, respectively $b=b_2$. Note that $\left(\mathbf{X}_{k / m_1}^{(1)}\right)=\left(X_k\right)_{k \in \mathbb{N}}$ and $\left(\mathbf{X}_t^{(2)}\right)_{t \geqslant 0}=\left(\overline{\mathbf{X}}_t\right)_{t \geqslant 0}$. Denote ${P}_t^{(1)}$ and ${P}_t^{(2)}$ the Markov semigroup associated with $\mathbf{X}_t^{(1)}$ and $\mathbf{X}_t^{(2)}$. Then for any $x \in \mathbb{R}^d$, $k\in \mathbb{N}^*$ we have
\begin{equation}
\label{single-error}
\boldsymbol{\delta}_x \mathrm{R}_{  1 / m_1 }^{km_1} = \boldsymbol{\delta}_x \mathrm{P}_k^{(1)}, \quad \boldsymbol{\delta}_x \mathrm{P}_{k} =  \boldsymbol{\delta}_x \mathrm{P}_k^{(2)} .
\end{equation}

From Lemma \ref{lemma:Lipschitz} and A \ref{Assumption3}, for any $t \in[j / m_1,(j+1) / m_1), j \in\{0, \ldots, m_1l - 1\}$ and $\left(w_t\right)_{t \in[0,l]} \in \mathrm{C}\left([0,l], \mathbb{R}^d\right)$ we have that 
\begin{equation}
\label{bound-b1-b2}
\begin{aligned}
\left\|b_1\left(t,\left(w_t\right)_{t \in[0,l]}\right)-b_2\left(t,\left(w_t\right)_{t \in[0,l]}\right)\right\|^2 
& =\left\|b_{\lambda}\left(w_{j / m_1}\right)-{b}_{\lambda}\left(w_t\right)\right\|^2 \\
& \leqslant \left(  \mathrm{L}_y+\alpha \mathrm{L} +1 / \lambda  \right)^2\left\|w_{j / m_1}-w_t\right\|^2.
\end{aligned}
\end{equation}
Using Cauchy-Schwarz inequality, Hölder's inequality and Itô's isometry we have for any $t \in[j / m_1,(j+1) / m_1)$,
\begin{equation}
\label{eq:error_from_sup_Ito}
\begin{aligned}
\mathbb{E}\left[\left\|\mathbf{X}_t^{(2)}-\mathbf{X}_{j / m_1}^{(2)}\right\|^2\right]
& = \mathbb{E}\left[\left\|  \int_{j / m_1}^t \left( b_\lambda\left(  \mathbf{X}_{\tau}^{(2)}  \right) \mathrm{d}\tau + \sqrt{2} \mathrm{~d} \mathbf{B}_\tau \right)\right\|^2\right] \\
& \leqslant \mathbb{E}\left[2 \left\| \int_{j / m_1}^t   b_\lambda\left(  \mathbf{X}_{\tau}^{(2)}  \right) \mathrm{d}\tau \right\|^2+2 \left\|\sqrt{2}   \left(\mathbf{B}_t - \mathbf{B}_{j/m_1} \right)\right\|^2 \right] \\
& \leqslant 2\left(t - \frac{j}{m_1}\right) \mathbb{E} \left[  \int_{j / m_1}^t   \left\| b_\lambda\left(  \mathbf{X}_{\tau}^{(2)}  \right) \right\|^2\mathrm{d}\tau \right]  + 4 d\left(t - \dfrac{j}{m_1}\right)  \\
& \leqslant 2\left(t - \frac{j}{m_1}\right)^2  \sup_{\tau\leqslant (j+1)/m_1}  \mathbb{E} \left\| b_\lambda\left(  \overline{\mathbf{X}}_\tau  \right) \right\|^2  + 4 d\left(t - \dfrac{j}{m_1}\right)
\end{aligned}
\end{equation}

Since we have proved \eqref{b-contractivity-condition}, \eqref{condition-lemma5.1}, \eqref{eq:Lipschitz_drift_lambda} in Appendix \ref{proof:contractive}, from Lemma 2.11 and Lemma 2.12 in \cite{majka2020nonasymptotic}, for any $\tau>0$ we have
\begin{equation}
\label{eq:bounded_expectation_X}
\mathbb{E} \left\| \overline{\mathbf{X}}_\tau  \right\|^2 \leqslant B_{0, 0},
\end{equation}
where $B_{0, 0}$ is an upper bound formed by $\lambda, C, b_\lambda(0), d, x$. Then from \eqref{eq:Lipschitz_drift_lambda} we have that
\begin{equation}
\label{eq:bound_drift_expectation}
\mathbb{E} \left\| b_\lambda\left(  \overline{\mathbf{X}}_\tau  \right) \right\|^2 \leqslant 2\left(\mathrm{L}_y+\alpha \mathrm{L} +1 / \lambda\right)^2 \mathbb{E} \left\| \overline{\mathbf{X}}_\tau \right\|^2 +2\left\| b_\lambda\left(  0  \right) \right\|^2 \leqslant B_{3}, \quad \forall \tau > 0,
\end{equation}
where $B_{3} = 2\left(\mathrm{L}_y+\alpha \mathrm{L} +1 / \lambda\right)^2 B_{0, 0} +2\left\| b_\lambda\left(  0  \right) \right\|^2 \geqslant 0$.

Then from (\ref{bound-b1-b2}), \eqref{eq:error_from_sup_Ito}, \eqref{eq:bound_drift_expectation},  for $i \in \{0, \cdots l - 1 \}$ we have that

\begin{equation}
\begin{aligned}
&\int_{i}^{i+1} \mathbb{E}\left[\left\|b_1\left(t, \mathbf{X}_t^{(2)}\right)-b_2\left(t, \mathbf{X}_t^{(2)}\right)\right\|^2\right] \mathrm{d} t \\
& \leqslant \sum_{j = i m_1 }^{(i+1)m_1 - 1} \int_{j / m_1}^{(j+1)/m_1} \mathbb{E}\left[\left\|b_1\left(t, \mathbf{X}_t^{(2)}\right)-b_2\left(t, \mathbf{X}_t^{(2)}\right)\right\|^2\right] \mathrm{d} t \\
& \leqslant m_1(\mathrm{L}_y+\alpha \mathrm{L} +1 / \lambda )^2  \int_{j / m_1}^{(j+1)/m_1} \mathbb{E}\left[\left\|\mathbf{X}_t^{(2)}-\mathbf{X}_{j / m_1}^{(2)}\right\|^2\right] \mathrm{d} t \\
& \leqslant (\mathrm{L}_y+\alpha \mathrm{L} +1 / \lambda )^2 \left( \frac{2B_{3}}{ 3m_1^2}+ \frac{2d}{m_1}  \right).
\end{aligned}
\end{equation}

From (\ref{single-error}) and Lemma SM6.1 in \cite{SupplementaryMaterials}, we obtain that there exists $B_b \geqslant 0$ such that for any $x \in \mathbb{R}^d$,
\begin{equation}
\label{error-1}
\begin{aligned}
& \left\|\boldsymbol{\delta}_x \mathrm{R}_{1 / m_1}^{lm_1}-\boldsymbol{\delta}_x \mathrm{P}_{l}\right\|_V 
 = \left\|\boldsymbol{\delta}_x \mathrm{P}_l^{(1)}-\boldsymbol{\delta}_x \mathrm{P}_l^{(2)}\right\|_{V} =  \left\| \boldsymbol{\delta}_x \mathrm{P}_l^{(2)} -  \boldsymbol{\delta}_x \mathrm{P}_l^{(1)}\right\|_{V}
\\
& \leqslant\left(\boldsymbol{\delta}_x \mathrm{P}_l^{(1)}\left[V^2\right]+\boldsymbol{\delta}_x \mathrm{P}_l^{(2)}\left[V^2\right]\right)^{1 / 2} \times \left(   \sum_{i = 0}^{l - 1}\int_i^{i+1} \mathbb{E}\left[\left\|b_1\left(t, \mathbf{X}_t^{(2)}\right)-b_2\left(t, \mathbf{X}_t^{(2)}\right)\right\|^2\right] \mathrm{d} t\right)^{1 / 2}
\\
& \leqslant  (\mathrm{L}_y+\alpha \mathrm{L} +1 / \lambda )  \sqrt{l\left( \frac{2B_{3}}{ 3m_1^2}+ \frac{2d}{m_1}  \right) } \left(\boldsymbol{\delta}_x \mathrm{P}_t^{(1)}\left[V^2\right]+\boldsymbol{\delta}_x \mathrm{P}_t^{(2)}\left[V^2\right]\right)^{1 / 2}.
\\
\end{aligned}
\end{equation}

Assume that there is a function $W \in \mathrm{C}^2\left(\mathbb{R}^d,[1,+\infty)\right)$ such that $\lim _{\|x\| \rightarrow+\infty} W(x)=+\infty$. Recall that from (\ref{condition-lemma5.1}),  using Lemma SM5.1  in \cite{SupplementaryMaterials} and  we get that there exist $\lambda_W\in \left(0, 1\right]$, $c,\beta\geqslant 0$ and $\zeta > 0$ such that for any $\delta \in \left(0, (4 \lambda)^{-1}\left(\mathrm{~L}_y+\alpha \mathrm{L} +1 / \lambda\right)^{-2}\right]$, $\mathrm{R}_\delta$ satisfies the discrete drift condition $\mathbf{D}_\mathrm{d}\left(  W, \lambda_W^\delta, c\delta  \right)$ and $\left(\mathrm{P}_t\right)_{t \geqslant 0}$ satisfies the 
continuous drift condition $\mathbf{D}_{\mathrm{c}}(W, \zeta, \beta)$. From Lemma SM5.2 in \cite{SupplementaryMaterials}, there exists $B_c \geqslant 0$ such that for any $x \in \mathbb{R}^d, t \geqslant 0$ and $k \in \mathbb{N}^*$ we have
\begin{equation}
\mathrm{R}_\delta^k W(x)+\mathrm{P}_t W(x) \leqslant B_c^2 W(x) .
\end{equation}
Let $W(x) = V^2(x)$ and $k = m_1l$, $\delta = 1/m_1$, $t = l$,  then  $ \forall x\in \mathbb{R}^d$,
\begin{equation}
\boldsymbol{\delta}_x \mathrm{P}_l^{(1)}\left[V^2\right]+\boldsymbol{\delta}_x \mathrm{P}_l^{(2)}\left[V^2\right] \leqslant B_c^2 V^2(x)
\end{equation}
Combined with (\ref{error-1}), we have that 
\begin{equation}
\left\|\boldsymbol{\delta}_x \mathrm{R}_{1 / m_1}^{m_1l}-\boldsymbol{\delta}_x \mathrm{P}_{l}\right\|_V \leqslant   B_c V(x)   (\mathrm{L}_y+\alpha \mathrm{L} +1 / \lambda )  \sqrt{l\left( \frac{2B_{3}}{ 3m_1^2}+ \frac{2d}{m_1}  \right) }
\end{equation}


To give a bound on $\left\| \boldsymbol{\delta}_x \mathrm{R}_{1/m_1}^{m_1l} - {\pi}_{\lambda}\right\|_V$, we use triangular inequality to split it into two terms: 
\begin{equation}
\begin{aligned}
& \left\| \boldsymbol{\delta}_x \mathrm{R}_{1/m_1}^{m_1l} - {\pi}_{\lambda}\right\|_V \leqslant 
  \left\| \boldsymbol{\delta}_x \mathrm{R}_{1/m_1}^{m_1l} - \boldsymbol{\delta}_x \mathrm{P}_l \right\|_V + \left\| \boldsymbol{\delta}_x \mathrm{P}_l -{\pi}_{\lambda}\right\|_V .
\end{aligned}
\end{equation}
Using this result and (\ref{error-P}), we obtain that there exists $B_1, B_2 \geqslant 0$ such that for any $m_1 \in \mathbb{N}^*$ with $ 1/m_1 \leqslant \bar{\delta} $,
\begin{equation}
\left\| \boldsymbol{\delta}_x \mathrm{R}_{1/m_1}^{m_1l} -  \pi_{\lambda}\right\|_V 
 \leqslant B_1 \tilde{\rho}_{1}^l V^2(x) + B_2 V(x)\sqrt{l\left( \frac{B_{3}}{ 3m_1^2}+ \frac{d}{m_1}  \right) } . 
\end{equation}
The proof in the general case where $\delta \in(0, \bar{\delta}]$ is similar when the interval $[0, l] $ is changed to $[0, lm_1\delta]$. 

Then we obtain that there exists $B_1, B_2, B_3 \geqslant 0$, $\tilde{\rho}_1\in [0, 1)$ such that for any $\delta \in(0, \bar{\delta}]$, $k\in \mathbb{N}^*$,
\begin{equation}
\left\| \boldsymbol{\delta}_x \mathrm{R}_{\delta}^{k} -  \pi_{\lambda}\right\|_V 
 \leqslant B_1 \tilde{\rho}_{1}^{k\delta} V^2(x) + B_2 V(x)  \sqrt{\delta^2 k \left(d + \dfrac{B_3 \delta}{3}\right)    } . 
\end{equation}

\end{proof}

\section{Additional experiments}\label{sec:CT-poisson}

The second limited-angle computed tomography reconstruction experiment we test is using the Poisson noise, where the model can be formulated as $y \sim P(Ax)$   and $P(Ax)$ denotes adding a  Poisson noise on $Ax$. We simulate the noisy sinogram as 
\begin{equation*}
y = -\frac{1}{\mu} \log \left(\frac{N_1}{N_0}\right),\quad N_1 \sim \operatorname{Poisson}\left(N_0 \exp ( - A(x) \mu)\right).
\end{equation*}
Here $N_0=4096$ is the mean photon count per detector bin without attenuation. $\mu=0.05$ is a  constant. Since Poisson noise implies a different likelihood
\begin{equation*}
\begin{aligned}
p(y| x) &= \dfrac{1}{K_{0}}\exp(- J(x, y)), \\
J(x, y) &=\sum_{i=1}^m e^{-A(x)_i \mu} N_0+e^{-y_i \mu} N_0\left(A(x)_i \mu-\log \left(N_0\right)\right),
\end{aligned}
\end{equation*}
we calculate $\nabla \log  p(y| x) = - \nabla J(x, y)$ by using the auto-gradient library.

We select a different $\alpha = 4000$ for NF-ULA while keeping all the other settings the same as in the main paper. 

\begin{figure}[!htbp]
    \caption{Limited-view CT reconstruction with Poisson noise. Column 1: Original image. Column 2: Filtered back projection (FBP). Columns 3, and 4: Posterior mean and the standard deviation of the samples generated by PnP-ULA (realSN-DnCNN). Columns 5, and 6: Posterior mean and the standard deviation of the samples generated by NF-ULA (patchNR). PSNR values of the sample mean images are provided in Table \ref{tab:CT_Poisson_noise_limited}.}
    \centering
    \includegraphics[width=0.95\linewidth]{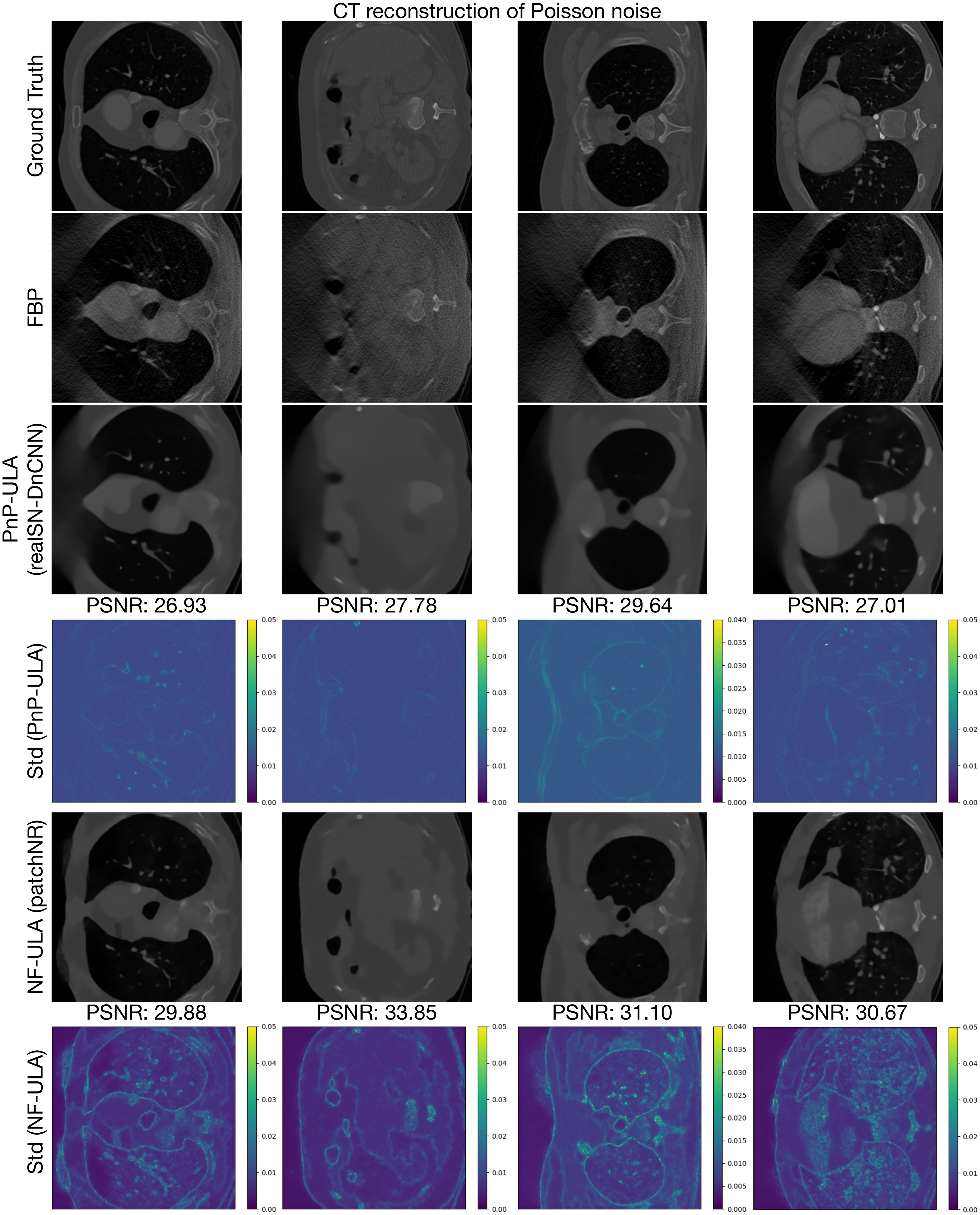}
    \label{fig:CT_Poisson_noise_limited} 
\end{figure}

Both PnP-ULA and NF-ULA have burn-in iterations of more than $ 20000 $. After the burn-in time, we calculate the posterior mean and the standard deviation by obtaining $10000$ samples and computing the PSNR of the samples' mean. For Poisson noise, the likelihood is more complicated than Gaussian, and NF-ULA spends 510s. 

Fig \ref{fig:CT_Poisson_noise_limited} includes the original image, the FBP, the posterior mean and the standard deviation of PnP-ULA (realSN-DnCNN) and NF-ULA (patchNR). Table \ref{tab:CT_Poisson_noise_limited} provides the PSNR of the posterior mean. All the samples generated in Table \ref{tab:CT_Poisson_noise_limited} never escape $ [-0.2, 1.2]^d $, indicating that the projection $\Pi_C(x)$ is never activated. Note that the huge uncertainties of standard deviation on the left area in the Gaussian-noise case in the main paper are slightly alleviated in the Poison noise experiments.  The ACF test results are similar to the CT experiment with Gaussian noise, therefore here we do not repeat them again.

\begin{table}[!t]
\caption{CT reconstruction of Poisson noise, limited angles.}
\centering
\begin{tabular}{|l|r|r|r|r|r|r|r|}
\hline
CT      & \multicolumn{3}{|l|}{ $C = [-100, 100]^d$ } \\
\hline
        & network       &  parameters        & PSNR     \\
\hline
figure1 &     \multicolumn{3}{|l|}{ } \\
\hline
NF-ULA  & PatchNR       & $  \alpha = 4000$ & 29.88   \\
\hline 
PnP-ULA & realSN-DnCNN  & $  \alpha = 3$    & 26.93    \\
\hline
figure2 &     \multicolumn{3}{|l|}{ } \\
\hline
NF-ULA  & PatchNR       & $  \alpha = 4000$ & 33.85   \\
\hline 
PnP-ULA & realSN-DnCNN  & $  \alpha = 3$    & 27.78    \\
\hline
figure3 &     \multicolumn{3}{|l|}{ } \\
\hline
NF-ULA  & PatchNR       & $  \alpha = 4000$ & 31.10   \\
\hline 
PnP-ULA & realSN-DnCNN  & $  \alpha = 3$    & 29.64    \\
\hline
figure4 &     \multicolumn{3}{|l|}{ } \\
\hline
NF-ULA  & PatchNR       & $  \alpha = 4000$ & 30.67   \\
\hline 
PnP-ULA & realSN-DnCNN  & $  \alpha = 3$    & 27.01    \\
\hline
\end{tabular}
\label{tab:CT_Poisson_noise_limited}
\end{table}

\end{document}